\documentclass{amsart}
\usepackage{amsfonts,amscd}

\newtheorem{theorem}{Theorem}[section]
\newtheorem{lemma}[theorem]{Lemma}
\newtheorem{corollary}[theorem]{Corollary}
\newtheorem{proposition}[theorem]{Proposition}
\theoremstyle{remark}

\theoremstyle{definition}

\newtheorem{example}[theorem]{Example}
\numberwithin{equation}{section} \makeatother

\DeclareMathOperator{\Cdb}{{\mathbb C}}
\DeclareMathOperator{\Rdb}{{\mathbb R}}

\DeclareMathOperator{\Ndb}{{\mathbb N}}

\begin{document}
\title[Projections on operator algebras]{Completely contractive projections on operator algebras}

\author{David P. Blecher}
\address{Department of Mathematics, University of Houston, Houston, TX
77204-3008}
\email[David P. Blecher]{dblecher@math.uh.edu}
 \author{Matthew Neal}
\address{Department of Mathematics,
Denison University, Granville, OH 43023}
\email{nealm@denison.edu}

\thanks{*Blecher was partially supported by a grant from
the National Science Foundation. Neal was supported
by Denison University.  }

\begin{abstract}
The main goal of this paper is to find operator algebra variants of 
certain deep results of St{\o}rmer, Friedman and Russo, Choi and Effros,
Effros and St{\o}rmer, Robertson and Youngson, Youngson,
and others, concerning projections on $C^*$-algebras and their ranges.  
 (See papers of these authors referenced in the bibliography.)    In particular 
we investigate the 
`bicontractive projection problem' and related questions
 in the category of operator algebras.  To do 
this, we will add the ingredient of `real positivity' from recent 
papers of the first author with Read.  
\end{abstract}

\maketitle

\section{Introduction}    
In previous papers both authors separately studied projections (that is,
idempotent linear maps) and conditional expectations
on unital operator algebras (that is, closed algebras of operators on a Hilbert space
that contain the identity operator) and other classes of Banach spaces.   (See papers of the authors referenced in the bibliography below.)  
Results were proved 
such  as: a completely contractive projection $P$ on such an algebra $A$ 
which is unital (that is, $P(1) = 1$) and whose range is a subalgebra, is a `conditional expectation'  (that is
$P(P(a)bP(c)) = P(a) P(b) P(c)$ for $a, b,  c \in A$)
\cite[Corollary 4.2.9]{BLM}.  This is an analogue of the matching theorem due to Tomiyama for $C^*$-algebras.
 
The main goal of our paper is to find variants, 
valid for operator algebras which are unital or which have an approximate identity,  of 
certain deeper results in the $C^*$-algebra case, due to St{\o}rmer, Friedman and Russo, 
Effros and St{\o}rmer, Robertson and Youngson,   
and others, concerning projections and their ranges.  
 (See papers of these authors referenced in the bibliography below.)     In particular 
we wish to investigate the 
`bicontractive projection problem' and related problems (such as the 
`symmetric projection problem' and the `contractive projection problem')
in the category of operator algebras.  
 To do this, we will add the ingredient of `real positivity' from recent 
papers of the first author with Read 
in \cite{BRI,BRII,BRord} (see also \cite{BNI,BNII,BBS,BOZ,B2015}),
A key idea in those papers is that  `real positivity' is often the right replacement in general algebras for positivity
in $C^*$-algebras.   This will be our  guiding principle here too. 

We now discuss the structure of our paper.  In Section 2
we  discuss completely contractive projections on operator algebras.
A well known and lovely result of 
Choi and Effros \cite[Theorem 3.1]{CE} (or rather, its proof), shows that
the range of a completely positive 
projection $P : B \to B$ on a $C^*$-algebra $B$, is again a $C^*$-algebra with
product $P(xy)$.   
 A quite deep theorem of Friedman and Russo, and a 
simpler  variant of this by Youngson, shows that something similar is
true if $P$ is simply contractive, or if $B$ is replaced by a  ternary ring of operators (\cite{CPP,Y}).    
The analogous result for unital completely contractive 
projections on unital operator algebras is true too, and is implicit in 
the proof of \cite[Corollary 4.2.9]{BLM} referred to above.
However there seems to be no analogous result 
for (nonunital) completely contractive 
projections on  nonunital operator algebras
without adding extra hypotheses on $P$.   
 The `guiding principle' in the previous paragraph suggests to add the condition that
$P$ is also `real completely positive' (we define this below).
Then the question does make good sense, and we are able to prove the 
desired result. 
Thus the range of a real completely positive completely contractive  projection 
$P : A \to A$ on an operator algebra 
with approximate identity, 
is again an operator algebra with
product $P(xy)$.  
We also have a converse and several complements to this result in Section 2,
as well as some other facts about completely contractive projections, such as how one is often
able to reduce the problem to
algebras which  have an identity.    We also show that for algebras with no kind of approximate identity, there is a biggest `nice part'  on which completely contractive projections (and the other classes
of projections discussed below) work well.

In Sections 3--5 we turn from the `contractive projection problem' to 
the `bicontractive projection problem' and related questions.        A projection $P$
is {\em bicontractive} if both $P$ and $I-P$ are contractive.
By the bicontractive projection problem for a Banach space $X$, one usually means one or both
of the following: 1)\ the 
characterization of all bicontractive projections  $P : X \to X$;
or 2)\  the
characterization of the ranges of the bicontractive projections.   
On a unital $C^*$-algebra $B$
it is known, by work of some of the authors mentioned above,  that
 the unital bicontractive projections are precisely $\frac{1}{2}(I + \theta)$, for a period $2$ $*$-automorphism
$\theta : B \to B$.   The possibly nonunital bicontractive projections $P$ on $B$ are of a similar form, and indeed if $P$ is also positive
then $q = P(1)$ is a central projection in the multiplier algebra $M(B)$ with respect to which $P$ decomposes into a direct sum of $0$ and
a projection of the above form $\frac{1}{2}(I + \theta)$, for a period $2$ $*$-automorphism
$\theta$ of $qB$. (See Theorem \ref{goq} for the idea of the proof of this.)
   Conversely, note that a map $P$ of the latter form is automatically  
`completely bicontractive' (that is, is bicontractive at each matrix level), indeed is `completely symmetric' (that 
is, $I - 2P$ is completely contractive), 
and the range of $P$ is a $C^*$-subalgebra,
and $P$ is a conditional expectation.     

One may  ask: what from the last paragraph is
true for general operator algebras?    Again the guiding principle referred to earlier
 leads us to use in place of the positivity in that result, the 
real positivity in the sense of \cite{BRII,BRord}.   
The next thing to note is that 
now `completely bicontractive' is   no longer  the same as `completely symmetric' for projections.
The `completely symmetric'  case works beautifully,  and the solution to our `symmetric  projection problem' is presented
in  Theorem \ref{trivch2} in Section 3. 
 This result
 is one somewhat satisfactory generalization (we shall see others later) to operator algebras 
of the $C^*$-algebraic theorem in the last paragraph.     For the more general class of 
completely bicontractive projections, 
as seems to be often the case in generalizing $C^*$-algebraic theory to more general algebras, 
a first look is disappointing.   Indeed most of the last paragraph no longer works in general.   One does not 
always get an associated completely isometric automorphism $\theta$ with $P = \frac{1}{2}(I + \theta)$, and
$q = P(1)$ need not be a central projection.  
Indeed we have solved here and elsewhere (see e.g.\ \cite{BL,Bare,BM} and see
also p.\ 92--93 in the prequel to the last cited paper, as well as the paper \cite{BNj} in preparation) many of
the obvious questions
about contractive projections, completely contractive projections, and
conditional expectations, on operator algebras.
Unfortunately many of the answers are counterexamples.
However, as also
seems to be often the case, a closer look at examples 
reveals an interesting question.   Namely, 
 given a real completely positive projection $P : A \to A$
which is completely bicontractive, when is the range of $P$ an
 (approximately unital)
 subalgebra of $A$, so that $P$ is a conditional expectation?   For operator algebras
we consider this to be the correct version of the `bicontractive projection problem'.
 In Sections 3 and 4 we elucidate this question.    We remark that in a paper in preparation \cite{BNj} 
we study the `Jordan algebra' variants of many of the results in Sections 2--3 of the present paper.
The Jordan variants of the material in Section 4 is not going to be
any better behaved than what we do there, and so we do not plan  discuss this much in \cite{BNj}.

  In Section 4 we discuss the completely bicontractive projection problem,
and construct some interesting examples, and give some 
reasonable conditions under which $P(A)$ is a subalgebra and $P$ is a conditional expectation.   
In particular we solve in full generality our version of the `bicontractive projection problem'
 for uniform algebras (that is, closed subalgebras of $C(K)$), and indeed for any algebra satisfying a condition
 related to semisimplicity.   Theorem 
\ref{mac} is one of the main results of the paper, giving a very general 
condition for $P(A)$ being a subalgebra
in terms of certain support projections.   In fact, at the time of writing, for all we know
the condition in Theorem
\ref{mac} is necessary and sufficient; at least  
we have no examples to the contrary.
 In Section 5 we discuss another condition that completely bicontractive projections may 
satisfy, and examine some consequences of this.
 In Section 6 we discuss Jordan homomorphisms on operator algebras, and amongst other things, solve two natural
`completely isometric problems', for Jordan subalgebras of operator algebras and for operator algebras, 
 related to the noncommutative Banach--Stone theorem.

We now turn to precise definitions and notation.    Any unexplained
terms below can probably be found in \cite{BLM,Paul,BRI,BRII}, or
any of the other books on operator spaces.  All vector spaces
are over the complex field $\Cdb$.   The letters $K, H$ denote  Hilbert spaces.
 If $X$ is an operator  space we often write  $X_+$ for the elements in $X$ which are positive (i.e.\ $\geq 0$)
in the usual
$C^*$-algebraic sense.
  We write ${\rm Ball}(X)$ for the set $\{ x \in X : \Vert x \Vert \leq 1 \}$.
By an  {\em operator algebra} we mean a 
not necessarily selfadjoint closed subalgebra of $B(H)$, the bounded operators  on a Hilbert space $H$.
We write $C^*(A)$ for the $C^*$-algebra generated by $A$, that is the smallest $C^*$-subalgebra containing $A$.
A {\em unital operator space} is a subspace $X$ of $B(H)$  or a unital $C^*$-algebra containing the identity (operator).
We often write this identity as $1_X$.  A map $T : X \to Y$ is {\em unital} if $T(1_X) = 1_Y$.  
We say that
an algebra is {\em approximately unital}
 if it has a contractive approximate
identity (cai).    
 For us a {\em projection in} an operator algebra  $A$ 
is always an orthogonal projection lying in $A$, whereas a {\em projection on} $A$
is a linear idempotent map $P : A \to A$.  If $A$ is a nonunital
operator algebra represented (completely) isometrically on a Hilbert
space $H$ then one may identify  the unitization $A^1$ with $A + \Cdb I_H$.   The
second dual $A^{**}$ is also an operator algebra with its (unique)
Arens product, this is also the product inherited from the von Neumann
algebra $B^{**}$ if
$A$ is a subalgebra of a $C^*$-algebra $B$.
Note that  
$A$ has a cai iff $A^{**}$ has an identity $1_{A^{**}}$ of norm $1$,
and then $A^1$ is sometimes identified with $A + \Cdb 1_{A^{**}}$.  The multiplier algebra $M(A)$ of such $A$ may be taken to be  
the `idealizer' of $A$ in $A^{**}$, that is $\{ \eta \in A^{**} : \eta A + A \eta \subset A \}$.

If $A$ is an approximately unital operator algebra or unital operator space then $I(A)$ denotes
the  injective envelope, an injective unital $C^*$-algebra containing $A$.  It contains $A$ 
as a subalgebra if $A$ is  approximately unital  \cite[Corollary 4.2.8]{BLM}.  For us the most important properties 
of $I(A)$ are, first, that it is injective in the category of operator spaces, so that 
any  from a subspace of an operator space $Y$ into $I(A)$ extends
to a complete contraction from $Y$ to $I(A)$.  Second, $I(A)$ is {\em rigid}, so that the  identity map on $I(A)$ 
is the only 
complete contraction $ : I(A) \to I(A)$ extending the identity map on $A$.  
The $C^*$-{\em envelope} $C^*_{\rm e}(A)$ is the $C^*$-subalgebra of $I(A)$ generated by $A$.   If $A$ is
unital it
has the property that given any unital complete isometry $T : A \to B(K)$, there exists a unique $*$-homomorphism
$\pi : C^*(A) \to  C^*_{\rm e}(A)$ with $\pi \circ T$ equal to the inclusion map of $A$ in  $C^*_{\rm e}(A)$.

We recall that a contractive completely positive map on a $C^*$-algebra is completely contractive.
A unital linear map between operator systems is positive and $*$-linear if it is contractive;
and it is completely positive iff it is  completely contractive.    See e.g.\ \cite[Section 1.3]{BLM} for these.  

A   {\em hereditary subalgebra}, or  HSA, in an operator algebra $A$  is
an   approximately unital subalgebra with $DAD \subset D$.  See \cite{BHN} for their basic theory.
The support projection of an HSA in $A$ is the identity of its bidual, viewed within $A^{**}$.

  We write ${\mathfrak r}_{A} = \{ x \in A : x + x^* \geq 0 \}$, and call these the {\em real positive elements}.  This space may be defined 
purely internally without using the `star', as the {\em accretive} elements, which 
have several purely metric definitions (see e.g.\ \cite[Lemma 2.4]{B2015}).
Also ${\mathfrak r}_{A}$ is the closure of the positive real multiples
of ${\mathfrak F}_A = 
\{ a \in A : \Vert 1 - a \Vert \leq 1 \}$ (see \cite{BRII}).      Read's theorem states that 
any operator algebra with cai has a  real positive cai (see e.g.\ \cite{Bnpi} for a proof of this).  
Since $M_n(A)^{**} \cong M_n(A^{**})$ (see \cite[Theorem 1.4.11]{BLM}),
we have that ${\mathfrak r}_{M_n(A^{**})}$ is the weak* closure 
of ${\mathfrak r}_{M_n(A)}$, and ${\mathfrak r}_{M_n(A)}
= M_n(A) \cap {\mathfrak r}_{M_n(A^{**})}$ for each $n$.

 A linear map $T : A \to B$ between operator
algebras or unital operator spaces is {\em real positive} if $T({\mathfrak r}_{A}) \subset {\mathfrak r}_{B}$.  It is {\em real completely
positive}, or {\em  RCP} for short, if $T_n$ is
real positive on $M_n(A)$ for all $n \in \Ndb$.   
One may also define these maps in terms of the 
set ${\mathfrak F}_A$ above \cite{BRII}, but the definitions
are shown to be equivalent in \cite[Section 2]{BBS}.
  From \cite{BBS}: a linear map $T : A \to B(H)$ on
an  approximately unital operator
algebra or unital operator space $A$ is   RCP iff $T$ has a completely positive (in the usual sense)
extension $\tilde{T} : C^*(A) \to B(H)$.   Here $C^*(A)$ is a $C^*$-algebra generated by $A$.
We call this the `generalized Arveson extension theorem'.
Thus real completely
positivity on 
$A$ is equivalent to $P$ extending to a completely positive map on a containing $C^*$-algebra.
A unital completely contractive map on a unital operator space is RCP, since it extends to a
completely contractive map on a containing unital $C^*$-algebra, and such maps are completely positive 
as we said above.

A TRO or {\em ternary ring of operators} is a subspace of $B(K,H)$ such that 
$Z Z^* Z \subset Z$.   A {\em WTRO} is a weak*
closed TRO.   The second dual of a TRO $Z$ is a WTRO (see \cite[Chapter 8]{BLM} for this
and the next several facts).
We write $L(Z)$ for the {\em  linking $C^*$-algebra} of a TRO, this has `four corners'
$Z Z^*$, $Z, Z^*,$ and $Z^* Z$.     Here $Z Z^*$ is the closure of the
linear span of products $z w^*$ with $z, w \in Z$, and similarly
for $Z^* Z$.     One gets a similar von Neumann algebra for WTRO's.
A {\em ternary morphism} on a TRO $Z$ is a linear map $T$ such that
$T(x y^* z) = T(x) T(y)^* T(z)$ for all $x, y, z \in Z$.
A {\em tripotent} is an element $u \in Z$ such that $u u^* u = u$.
We order tripotents by $u \leq v$ if and only if $u v^* u = u$.
This turns out to be equivalent to $u = v u^* u$, or to $u =  u u^* v$,
and implies that $u^* u \leq v^* v$ and  $u u^* \leq v v^*$ \cite{Bat}.
If $x  \in {\rm Ball}(Z)$,  define $u(x)$ to
be the weak* limit of the sequence $(x (x^{*}x)^n)$ in $Z^{**}$.
This is the largest tripotent in $B^{**}$ satisfying
$v v^* x = v$ (see \cite{ER4}).     If $x \geq 0$, or if $u(x)$ is a projection
then   $u(x)$ is also the weak* limit of powers
$x^n$ as $n \to \infty$ (see e.g.\ \cite{BNII,Bnpi}).

We will say that an idempotent linear $P : X \to X$ is a
{\em symmetric} (resp.\ {\em completely symmetric}) projection, if 
$\Vert I - 2P \Vert \leq 1$ (resp.\ $\Vert I - 2P \Vert_{\rm cb}  \leq 1$).   This is related to 
the notion of $u$-ideal \cite{GKS}, but we will not need anything from that theory. Such 
are automatically bicontractive  (resp.\ completely bicontractive).
We say that $P$ is {\em completely hermitian} 
if $P$
 is hermitian in $CB(X)$. 
   Note that since exp$(itP) = I - P + e^{it} P$  
it follows that $P$ is 
completely hermitian iff
$\Vert I - P + e^{it} P \Vert_{\rm cb} \leq 1$ for all real $t$.
This is essentially the notion of being  (completely)  `bicircular'.
Clearly if $P$ is completely  hermitian then it is completely symmetric.  
We will not discuss (completely) hermitian projections much in this paper, these seem
much less interesting. 

\section{Completely  contractive projections on approximately unital operator algebras}  

Looking at examples it becomes clear that projections on operator algebras with no kind of approximate identity
can be very badly behaved.   Hence we will say little in our paper about 
such algebras.  However it is worth mentioning that we can pick out a `good part' of such a projection.
This is the content of our first result.

\begin{proposition} \label{reau}  Let $P : A \to A$  be a real
completely positive completely contractive map (resp.\  projection)
on an operator algebra $A$ (possibly with no kind of approximate identity).
There exists a largest approximately unital subalgebra $D$ of $A$,
and it is a   HSA (hereditary subalgebra) of $A$.
Moreover, $P(D) \subset D$, and the restriction $P'$ of $P$ to $D$ is a 
 real completely positive completely contractive map (resp.\  projection)
on $D$.   In addition,  $P'$ is completely bicontractive (resp.\  completely symmetric) if 
$P$ has the same property.  \end{proposition}  

\begin{proof}     By \cite[Corollary 2.2]{BRord}, $D = {\mathfrak r}_A - {\mathfrak r}_A$ is 
the largest approximately unital subalgebra of $A$.   This algebra is written as $A_H$ there, and was first introduced in \cite[Section 4]{BRII}.
Clearly $P(D) \subset D$.   The rest is obvious.
 \end{proof}

{\bf Remarks.}    1) \ The last result is also true the word `completely' removed throughout, with the same
proof.

\medskip

2) \  Letting $p$ be the support projection of the HSA $D$ above,  if $P$  extends to a 
completely positive complete contraction on a containing $C^*$-algebra (as in our 
generalized Arveson extension theorem mentioned in the introduction, see also \cite[Theorem 2.6]{BBS}) then one can show
that $P^{**}(pa) = pP(a)$ and $P^{**}(ap) = P(a) p$ for $a \in A$.   It follows that $P$ may be 
pictured as a $2 \times 2$ matrix with its `good part' above in the $1$-$1$ corner.   However in general it seems
one can say little about the other corners, they can be quite messy.   This is why we focus on 
algebras with approximate identities in our paper.

\begin{proposition} \label{cepro} A  real
completely positive completely contractive map (resp.\  projection)
on an approximately unital operator algebra $A$, extends to a unital (real
completely positive) completely contractive map (resp.\   projection)
on the unitization $A^1$.  (If $A$ is unital 
then we define $A^1 = A \oplus^\infty \Cdb$ here.)
\end{proposition}  \begin{proof}   Suppose that $P : A \to A \subset B(H)$ is the projection.
By \cite[Theorem 2.6]{BBS}, $P$ extends uniquely to a
completely positive completely contractive map $C^*(A) \to B(H)$. 
 By \cite[Lemma 3.9]{CECPL} it extends further
to a unital completely positive map $C^*(A) + \Cdb I_H \to B(H)$.   The restriction of the latter map to $A + \Cdb I_H$ may be
viewed as a unital (real
completely positive) completely contractive map
on the unitization $A^1 \to A^1$, and it is evidently a projection if $P$ was a projection.  \end{proof}

The previous result gives a way to `reduce to the unital case'.   However this method does not seem to be helpful later in our paper when dealing with 
bicontractive or symmetric projections, and we will need a different `reduction to the  unital case'.

\begin{lemma} \label{ispos}  Let $P : A \to A$ be a real positive contractive
map on a unital operator algebra.  Then $0 \leq P(1) \leq 1$.
\end{lemma}

\begin{proof} 
The restriction of $P$ to $\Delta(A) = A \cap A^*$ is
real positive.  Hence it is positive by the proof of
\cite[Theorem 2.4]{BBS}.  So  $0 \leq P(1) \leq 1$.
\end{proof}

\begin{lemma} \label{ceproj}  Suppose that $E$ is a completely contractive completely positive projection on an operator system $X$.  Then the range of $E$,
with its usual matrix norms,
 is an operator system with matrix cones $E_n(M_n(X)_+) = M_n(X)_+ \cap 
{\rm Ran}(E_n)$, 
and unit $E(1)$.
\end{lemma}

\begin{proof}  We will use the Choi-Effros characterization of operator systems
\cite{CE}.  Because ${\rm Ran}(E_n)$ is a $*$-subspace 
of $M_n(X)$,  with the inherited cone from $M_n(X)_+$, it is a partially ordered, matrix ordered,
Archimidean $*$-vector space with  proper cones.   If $x = x^*$
there exists a positive scalar $t$ with $-t 1 \leq x \leq t 1$, so that
$-t E(1) \leq x \leq t E(1)$.  So $E(1)$ is an order unit.
If $x \in {\rm Ran}(E)$ with $\Vert x \Vert_X \leq 1$ then  
$$\left[ \begin{array}{ccl}
1 & x \\
x^*   & 1
\end{array}
\right] \; \geq \; 0 $$
in $M_2(X)$.  Applying $E_2$, we deduce that 
$$\left[ \begin{array}{ccl}
E(1) & x \\
x^*   & E(1) 
\end{array}
\right] \; \geq \; 0,$$
so that the norm of $x$ is $\leq 1$ in the new `order unit norm'
(see \cite[p.\ 179]{CE}).  
Conversely, if the last norm is $\leq 1$,
or equivalently if the last centered equation holds,
then it is a simple exercise in operator theory that
$\Vert x \Vert_X \leq 1$, since $\Vert E(1) \Vert \leq 1$ and $E(1) \geq 0$.   
Thus the `order unit norm' coincides with the old norm.
Similarly at each matrix level.
By the Choi-Effros characterization of operator systems       
$({\rm Ran}(E), E(1))$ is 
an operator system with the given matrix cones,
and the order-unit matrix norms are the usual norms.
\end{proof} 

The following generalization of the Choi-Effros result 
referred to in the introduction, solves the
`completely contractive projection problem' in the category of approximately unital operator algebras and real completely
positive projections.   We remark that the case when also $P(1_A) = 1_A$ is implicit 
in the proof of \cite[Corollary 4.2.9]{BLM}.    

\begin{theorem} \label{tr}  Let $A$ be an 
approximately unital operator algebra,
and $P : A \to A$ a completely contractive projection which is
also real completely positive.
Then 
the range $B = P(A)$ is an
approximately unital operator algebra with product $P(xy)$.   
We have $$P(P(a) b) = P(P(a) P(b)) = P(a P(b)), \qquad a, b \in A.$$ 
In particular $P(P(1)^n)) = P(1)$ for all 
$n \in \Ndb$, if $A$ is unital.        With respect to 
the `multiplication' $P(xy)$,  
$A$ is an  bimodule over $B$ , and $P$ viewed as
a map $A \to B$ is a $B$-bimodule map
(with $B$ equipped with its new  product).  
If $A$ is unital then $P(1)$ is the
 identity for  the latter product.  Moreover (not assuming $A$ unital), $P$ extends
to a completely positive completely contractive projection on the injective envelope
$I(A)$. 
\end{theorem}

\begin{proof}  We give two proofs, since they both use techniques the reader will
need to be familiar with in the rest of our paper.

Set $B = P(A)$.
Extend $P$ to a unital completely contractive projection $P^1$
 on $A^1$ by Proposition \ref{cepro}.   We may then use the 
proof of \cite[Corollary 4.2.9]{BLM}, which proceeds by extending 
$P$ to a unital (completely positive and)
 completely contractive projection $E$ on
$I(A^1)$.   It follows from the Choi-Effros result mentioned early in our introduction,
that Ran$(E)$ is a unital $C^*$-algebra with product
$E(xy)$, and $B$ with product $P(xy)$ is a unital  subalgebra
of this $C^*$-algebra.  We also have by the same Choi-Effros result
(or its proof) that $E(E(a) b) = E(E(a) E(b)) = E(a E(b))$
for all $a, b \in A$, giving the centered equation in the theorem statement.
This gives the first several
assertions of our theorem.  Note that 
$$P(P(e_t) P(a)) = P(e_t P(a)) \to P(P(a)) = P(a), \qquad a \in A, $$
if $(e_t)$ is a cai for $A$.  Similarly on the right, so that 
$(P(e_t))$ is a cai for $P(A)$ in its new product.   
That $A$ is actually a $B$-bimodule follows from the centered equation  in the theorem statement;  for example  because
$P(P(P(a)P(b))c)$ equals
$$P(P(P(a)P(b)) P(c)) = P(P(a)P(b) P(c)) =  P(P(a)P(P(b) P(c))) =  P(P(a)P(P(b)c)).$$
The centered equation in the theorem statement
is just saying  that $P$ is a $B$-bimodule map for the given products.
The final assertion about extending to $I(A)$ is easy from the above
in the case that $A$ is unital; the other case we will do below.   

For the second proof, first suppose that $A$ is unital.   
Let $B = P(A)$, and set $X = A + A^*, Y = B + B^*$ and $v = P(1)$.
By \cite[Theorem 2.6]{BBS}, $P$ extends uniquely to a 
completely positive completely contractive map $P'$ on $X$.
Since $X = A + A^*$ this map is uniquely determined,
it must be $a_1 + a_2^* \mapsto P(a_1) + P(a_2)^*$, 
a projection on $X$ with range $Y$. 
By Lemma \ref{ceproj} $(Y,v)$ is an operator system with positive cone
$P(X_+)$.  
Let  $j : Y \to X$ be the inclusion map.  Then 
extend $P'$ to a 
completely positive complete contraction $\tilde{P} : I(A) \to I(B)$ 
by \cite[Theorem 2.6]{BBS}.  Extend $j$ to a 
completely positive complete contraction
$\tilde{j} : I(B) \to I(A)$  
by Arveson's extension theorem \cite{Ar}.   Then $\tilde{P} \circ \tilde{j}$
equals the identity map on $I(B)$ by rigidity of the injective
envelope, since $P \circ j = I_B$.  Thus $E = \tilde{j} \circ \tilde{P}$
is a completely positive completely contractive
projection on $I(A)$ extending $P$.   We deduce just as in the last paragraph that
Ran$(E)$ is a unital $C^*$-algebra with product
$E(xy)$, $B$ with product $P(xy)$ is a unital  subalgebra
of this $C^*$-algebra, and 
 $P(P(a) b) = P(P(a) P(b)) = P(a P(b)).$

Finally, if $A$ is nonunital but  
approximately unital, then $P^{**}$ is
a completely contractive projection on $A^{**}$, which is
also real positive by the proof of the main theorem in
\cite[Section 2]{BBS}.  By the unital case,
$P^{**}(xy)$ is an operator algebra product on
$P^{**}(A^{**})$, with unit $v = P^{**}(1)$.   Hence by restriction
$P(xy)$ is an operator algebra product on
$P(A)$, and the centered equation in the theorem statement holds on $A$, as does the 
assertions about bimodules.
Note that $P^{**}(A^{**}) = 
(P(A))^{\perp \perp}$, so that $P^{**}(A^{**}) \cong P(A)^{**}$.
So $P(A)^{**}$ is unital, and hence $P(A)$ is approximately unital (or this may be seen directly
using the centered equation in the theorem statement).
  Also, since $v = P^{**}(1)$ acts as an identity on
$P(A)$ in the new product, 
we can identify $P(A) + \Cdb v$, as a unital operator space,
 with the 
unitization of $P(A)$ with its new operator algebra product.
Then the restriction $r$ of $P^{**}$ to $A + \Cdb 1_{A^{**}}$
can be viewed as a real completely positive completely contractive 
projection on the unitization $A^1$.
By the last paragraph, $r$ extends to a 
completely positive completely contractive
projection on $I(A^1)$.  However, $I(A^1) = I(A)$
by e.g.\ \cite[Corollary 4.2.8]{BLM}.
\end{proof}

{\bf Remarks.} 
1) \ Thus the category of approximately unital operator algebras and real completely
positive projections forms a `projectively stable category' in the sense of Friedman and
Russo (see Neal and Russo \cite[p.\ 295--296]{NR}, and e.g.\ \cite{CPP}).  Namely, if ${\mathcal B}$
is the category of Banach spaces with morphisms being the
contractive projections, a subcategory ${\mathcal S}$ of ${\mathcal B}$
is {\em projectively stable} if ${\mathcal S}$ is closed under images of morphisms.
That is, for an object $E$ and morphism $\varphi : E \to E$ in ${\mathcal S}$, $\varphi(E)$ is
again an object in ${\mathcal S}$ (although not necessarily a `subobject' with respect
to the full structure of objects in ${\mathcal S}$).
For example, the subcategory of unital $C^*$-algebras and completely positive unital
projections is projectively stable, by the theorem of Choi and Effros used earlier.   Other 
projectively stable categories are listed in the last references; e.g.\ 
  the subcategory of TRO's and completely  contractive projections
is projectively stable by Youngson's theorem \cite{Y}.     In the cited pages of \cite{NR}
the concept of a `projectively rigid category' is discussed.  The associated question for us
would be if the preduals of dual operator algebras are projectively (completely) rigid in their sense.
However the answer to this is in the negative, since the category of Banach or operator spaces
is not projectively (completely) rigid, and then one can play the ${\mathcal U}(X)$ trick (described
above Proposition \ref{one}  below)  to 
answer the operator algebra question.

\medskip
  2) \ 
If $A$ is unital and $C$ is the 
$C^*$-subalgebra of $I(A)$ generated by $P(A)$, then the 
map $\tilde{P}$ in the proof restricts to a 
$*$-homomorphism from $C$ onto $C^*_{\rm e}(B)$, the latter viewed as a
subalgebra of $I(B)$ (or as a $C^*$-subalgebra of the space Ran$(E)$ in 
the proof, with its `Choi-Effros product').   See e.g.\ 
\cite[Theorem 1.3.14 (3)]{BLM}.   

\medskip

\begin{lemma} \label{zelo}  Let $A$ be a unital operator algebra,
and $P : A \to A$ a contractive projection, such that ${\rm Ran}(P)$ contains an orthogonal  
projection $q$ with $P(A) = q P(A) q$.   Then $q = P(1_A)$.
\end{lemma}

\begin{proof}   We have $\Vert q \pm (1-q) \Vert \leq 1$, so that $\Vert q \pm P(1-q) \Vert \leq 1$.
Since $P(A) = q P(A) q$, and $q$ is an extreme point of the 
unit ball of $qAq$ (the identity is an extreme point of the 
unit ball of any unital Banach algebra), we have that $P(1-q) = 0$.  Thus $P(1) = q$. 
\end{proof}  

The following `reduction to the case of unital maps'  works under a certain condition which will be seen to be
automatic in the setting found in 
the next Sections of the paper.

\begin{proposition} \label{tr2}  Let $A$ be an approximately 
unital operator algebra,
and $P : A \to A$ a completely contractive projection.   Then ${\rm Ran}(P^{**})$ contains an orthogonal 
projection $q$ such that $P(A) = q P(A) q$ iff  $P^{**}(1)$ is a projection.  
In this case $q = P^{**}(1)$, and ${\rm Ran}(P)$ is an approximately
 unital operator algebra with product $P(xy)$, and its bidual has identity $q$.
Also, $P$ is real completely positive, all the conclusions of Theorem {\rm \ref{tr}} hold, $q$ is an open projection 
for $A^{**}$ in the sense of {\rm  \cite{BHN}}, and
$$P(a) = q P(a)q  = P^{**}(qaq), \qquad a \in A,$$ (and we can replace $P^{**}$ by $P$ here if $A$
is unital).   Hence $P(A) = q P(A) q = P^{**}(qAq)$, and 
$P$ `splits' as the sum of  the zero map on $q^\perp A + A q^\perp + q^\perp A q^\perp$, and
a real completely positive completely contractive projection $P'$ on $qAq$ with range equal to 
$P(A)$.  This  projection $P'$ on $qAq$ is unital if $A$ is unital.   
\end{proposition}

\begin{proof}  
 Let $Q = P^{**}$,  a completely contractive projection  on $A^{**}$.  We can replace $Q$ by $P$ below
if $A$ is unital.   
If $P(A) = q P(A) q$ for a projection $q$ then $Q(A^{**}) = q Q(A^{**})  q$ by standard weak* approximation
arguments, so by the lemma, $Q(1) = q$.   Conversely, suppose that  $Q(1) = q$ is a projection.  Then 
$Q(q^\perp) = 0$.    Note that Ran$(Q)$ is a unital operator space (in $qA^{**}q$).  
So $Q$, and hence also $P$,  is real completely positive by \cite[Lemma 2.2]{BBS}, 
since it extends by e.g.\ \cite[Lemma 1.3.6]{BLM} to a completely positive unital map from $X + X^*$ onto $Y + Y^*$
where $X = A^{**}$ and $Y = Q(A^{**})$.  By extending $Q$ further to a completely positive  completely contractive  map on  a containing
$C^*$-algebra, and using the Kadison-Schwarz inequality, we have
$$Q(a q^\perp)^* Q(a q^\perp) \leq Q(q^\perp a^*a q^\perp) \leq Q(q^\perp) = 0, \qquad a \in {\rm Ball}(A^{**}).$$
Thus $Q(a) = Q(aq)$ for all $a \in A^{**}$, and similarly $Q(a) = Q(qa)$.  
Also $Q(q)^2 = Q(q)$, and so $P(A) =  Q(qAq) =  q P(A) q$ by Choi's multiplicative
domain trick (the latter is usually stated for unital maps, but the general case may be reduced to this
using \cite[Lemma 3.9]{CECPL}).

The rest follows from Theorem \ref{tr}
and its proof, with the exception of  $q$ being an open projection 
for $A^{**}$.   To see this,
if  $(e_t)$ is a cai for $P(A)$ with its new product  then using 
some of the facts here and in Theorem \ref{tr}  we have
$e_t = P(e_t q) = P(e_t) q \to q$ weak*.  
\end{proof}

{\bf Remarks.} 1) \ Note that even  a completely contractive completely positive projection on a unital $C^*$-algebra need not have
$P(1)$ a projection.  To see this, choose a  norm $1$ element $x \neq 1$ in $A_+$ and a state $\varphi$ with 
$\varphi(x) = 1$, and consider $P = \varphi(\cdot) \, x$.

\medskip

2) \   Unfortunately the projection $q$ here need not be central, even if $P$ is completely bicontractive.  See the next example.  

\begin{example} \label{exoo}  Consider  the canonical projection of the upper triangular 
matrices $A$  onto $\Cdb E_{11}$.  This is a completely bicontractive projection
(which is also completely bicontractive, completely hermitian, etc), but it does not extend to a
completely bicontractive projection on its $C^*$-envelope (or 
injective envelope) $M_2$.  In this case note that  $A+A^* = C^*(A) = C^*_{\rm e}(A) = I(A)$.
On the positive side, the range of this projection 
is a subalgebra of $A$.   \end{example}

\begin{corollary} \label{acp}  Let $A$ be an approximately  
unital operator algebra, with an approximately  
unital subalgebra $B$ which is the range of a completely contractive projection $P$ on $A$.
Then  $P$ is real completely positive, and all the conclusions of Theorem {\rm \ref{tr}} hold.
Hence $P$ is a conditional expectation: $P(a)b = P(ab)$ and $bP(a) = P(ba)$
for all $b \in B =  P(A)$ and $a \in A$.  It follows that 
$(P(e_t))$ is a cai for $B$ for any cai $(e_t)$ of $A$.   
\end{corollary}

\begin{proof}   Consider $P^{**}$,  a completely contractive projection on $A^{**}$ with 
range $B^{**}$.  Of course $B^{**}$ has an identity of norm $1$ as we said in the introduction.
 By Proposition \ref{tr2},   $P^{**}$  is real completely positive, and hence so is its
restriction $P$.  
   The remaining assertions 
follow e.g.\ from Proposition \ref{tr2}, except for the last assertion
which is an easy consequence of
 the second last assertion.   \end{proof}

The last result, which may be seen as a converse to Theorem \ref{tr}, generalizes
the fact from \cite{BLM} mentioned at the start of the introduction.  
We showed in \cite{Bare} that this is all false with the word `completely' 
removed, however see \cite{LL} for some later variants
valid for certain Banach algebras.

We will need the following results later.
If $P : M \to M$ is a unital completely contractive projection on a von Neumann algebra,
there exists a {\em support projection} $e$, the perp of the supremum of all projections
in Ker$(P)$, as in {\rm \cite[p.\ 129]{ES}}.  We have $e \in P(M)'$, and $P(x) = P(ex) = P(xe)$ for all $x \in M$, and if $x \in M_+$ then 
$P(x) = 0$ iff $xe = 0$ iff $ex = 0$ (see e.g.\ around Lemma 1.2 in \cite{ES}).     Following the idea in 
the proof of \cite[Lemma 1.2 (3)]{ES} we have:

\begin{proposition} \label{ES}  Let $P : M \to M$ a weak* continuous
unital completely contractive projection on a von Neumann algebra $M$.    Let $e$ be the support projection of $P$ on $M$ discussed above.   Let $N$ be the von Neumann algebra generated by $P(M)$.
Then $P(x) e = eP(x)e = exe = xe$ for all $x \in N$. \end{proposition}

\begin{proof}    For $n = 0, 1, \dots$, let $A_n$ be the span of products of $2^n$ elements from $P(M)$.
Then $A_n$ is a unital $*$-subspace of $M$.   Suppose that  $eP(x)e = exe$ for all $x \in A_n$.   Then for such $x$, set $z = e(P(x^* x) - x^* ex)e$.  Following the steps in the 
 proof of \cite[Lemma 1.2 (3)]{ES} with minor modifications we have $P(z) = 0$ and $z \geq 0$, so that
by the facts above the present proposition we obtain 
$z = eze = 0$ and $eP(x^* x)e = ex^* xe$.  By the polarization identity
$eP(y^* x)e = ey^* xe$ for $x, y \in A_n$.  So  $eP(x)e = exe$ for all $x \in A_{n+1}$, and hence
for all $x \in N$. \end{proof}  

\begin{corollary} \label{newaei}   Let $P : A \to A$ be a unital completely contractive projection on an operator algebra.
If $P(A)$ generates $A$ as an operator algebra, then $(I-P)(A) = {\rm Ker}(P)$ is an ideal in $A$.  In any case,
if $D$ is the closed algebra generated by $P(A)$ then $(I-P)(D)$ is an ideal in $D$.
 \end{corollary}

\begin{proof} 
 We may assume that $P(A)$ generates $A$.
As above we extend $P$ to a unital completely contractive projection $\tilde{P}$ on a $C^*$-algebra $B$
($= I(A)$).  The second adjoint of this is a weak* continuous unital completely contractive projection  on a von
Neumann algebra $M$, and we continue to write this projection as $\tilde{P}$.    Let  $\tilde{P}$ also 
denote the restriction of the latter projection to the von Neumann algebra $N$ generated by $P(A)$ inside $M$.
  If $x \in (I-P)(A)$, then $\tilde{P}(x) = 0$, and so by Proposition \ref{ES} we have
$ex = x e = 0$.  Thus $x \in e^{\perp} M e^{\perp}$ (and is also in
$e^{\perp} N e^{\perp}$).   So for  $y \in A$ we have 
$P(xy) = P(exy) = 0$.   Similarly $yx \in (I-P)(A)$, so the latter is an ideal.  \end{proof}  

For later use we record that  if $P : A \to A$ is a unital completely contractive projection on an operator algebra, then
 in the language of the last proof,  $$D \cap e^\perp M e^\perp  = 
D \cap e^\perp N e^\perp  = {\rm Ker}(P_{|D}) = (I-P)(D) .$$
Indeed since we said above Proposition \ref{ES}  that $P(x) = P(exe)$ for any $x \in M$, we have 
 $D \cap e^\perp M e^\perp  \subset {\rm Ker}(P_{|D})$.   Moreover, 
if $d \in D$ with $P(d) = 0$ then 
the argument in the last proof 
shows that $d \in D \cap e^\perp N e^\perp \subset D \cap e^\perp M e^\perp$.
So $D \cap e^\perp M e^\perp  = D \cap e^\perp N e^\perp =  {\rm Ker}(P_{|D})$.

\section{The symmetric  projection problem and the bicontractive projection problem}

It turns out that the variant of the bicontractive projection problem 
for symmetric projections works out perfectly.
This is the question of characterizing (completely) symmetric projections in the categories
we are interested in, and their ranges.
Notice that if $P : X \to X$ is a projection on a normed space
and if we let $\theta = 2P-I$, so that $P = \frac{1}{2} (I + \theta)$, then
Ran $(P)$ is exactly
the set of fixed points of $\theta$, and  $\theta \circ \theta = I$.
Note too that $\theta$ is contractive if $P$ is
symmetric.
From the latter facts 
we deduce  that $\theta$ is a bijective isometry whose inverse is itself.
Also  $\theta(1) = 1$ if $X$ is a unital algebra and $P(1) = 1$.     Applying the same argument
at each `matrix level' we see that:

\begin{lemma} \label{sostate}  
A projection  $P : X \to X$ on an operator space is completely symmetric (resp.\  symmetric) iff 
 $\theta = 2P-I$ is a complete isometry  (resp.\ an isometry), and in this case $\theta$ is also a surjection.   If $X$ is also an algebra (resp.\  Jordan algebra) and 
 if $\theta$ is a homomorphism (resp.\  Jordan homomorphism)
then the range of $P$ is a subalgebra  (resp.\  Jordan subalgebra). 
 \end{lemma}

\begin{proof}  For the last part, Ran $(P)$ is exactly
the fixed points of $\theta$.    
  \end{proof}

Thus the (completely) `symmetric projection problem'
in some sense a special case of the (complete) `isometry problem':
namely characterizing the linear (complete) isometries 
between the objects in our category.   That is,
the key to solving the (completely) `symmetric projection problem'
is proving a `Banach--Stone type theorem' in our category.   The original
Banach--Stone theorem characterizes unital isometries between $C(K)$ spaces,
and in particular shows that such are $*$-isomorphisms.  Putting this together 
with the last assertion of the last lemma,  we see that one of the hoped-for conclusions of the 
(completely) `symmetric projection problem', and by 
extension the (completely) `bicontractive projection problem', 
is that the range of the projection is a subalgebra.    
We will also show  in the completely symmetric case that 
 if $A$ is unital or approximately unital then so is $P(A)$. 

Let us examine what this all looks like in a $C^*$-algebra, where as predicted in
the last paragraph, much hinges on the known `Banach--Stone type theorem'
for $C^*$-algebras, due mainly to Kadison.  The following 
is essentially well known (see e.g.\ \cite{FRAD,St}), but 
we do not know of a reference which has all of these assertions, or is in the formulation  we give:  

\begin{theorem} \label{goq}  If $P : A \to A$ is a    
projection on a $C^*$-algebra $A$ then 
$P$ is bicontractive iff $P$ is symmetric. 
Then $P$ is bicontractive and completely 
positive iff  there exists a central projection $q \in M(A)$, such that
 $P = 0$ on $q^\perp A$, and there exists a period $2$ $*$-automorphism $\theta$ of $qA$ 
so that $P = \frac{1}{2}(I + \theta)$ on $qA$.    
\end{theorem}

\begin{proof}   
Clearly symmetric projections are bicontractive.
Conversely, if $P$ is bicontractive then by \cite[Theorem 2]{FRAD}
 $\theta =2P - I_A$ is a linear surjective isometric Jordan triple product preserving selfmap of $A$
with $\theta \circ \theta = I_A$, and $P = \frac{1}{2}(I_A + \theta)$.  So $P$ is symmetric.  If also $P$ is positive then $P$, and hence also $\theta$, is $*$-linear.      Let $Q$ be the extension of $P$ to the second dual.

Suppose that  $\theta : A \to B$ is a linear isometric surjection between
$C^*$-algebras.  By a result of Kadison \cite{Kad1},  $u = \theta^{**}(1)$ is a  unitary in $B^{**}$.
Suppose now further that
$\theta$ is $*$-linear.  Then $u$  is selfadjoint, 
and $u \theta^{**}(\cdot)$ is a unital 
isometry so selfadjoint.  Thus $u \theta(a)= \theta(a^*)^* u = \theta(a) u$, for $a \in A$, 
so $u$ is   central.    
 If $a \in A_{\rm sa}$ then $$u\theta(a^2) = (u\theta(a))^2  = \theta(a)^2 \in B,$$
since $u \theta(\cdot)$ is a Jordan morphism.  Thus $uB = u \theta(A)  \subset B$.
So $u \in M(B)$.

Returning to our situation, let $q = Q(1)$.    This is a central projection in $M(A)$, since $u = 2q-1$.       Since $Q(q^\perp) = 0$, if $a \in {\rm Ball}(A)_+$ then
$$P(q^\perp a) = P(q^\perp a q^\perp) \leq Q(q^\perp) = 0,$$
and so $P = 0$ on $q^\perp A$.  Also, since $\theta(q) = q$ and $\theta$ is
 Jordan triple product preserving,
it follows that $P(qa) = \frac{1}{2}(qa + \theta(qa)) = q P(a).$
Thus,
$P(qA) = qP(A)$, and
 the restriction of $P$ to $qA$ is a unital bicontractive positive projection on a unital  $C^*$-algebra.  
Also $\theta(qa) = q \theta(a)$ for $a \in A$, as we had above,
 so $\theta(qA) = q A$.  Hence $\theta' = \theta_{\vert qA}$ is 
a  unital isometric isomorphism of $qA$.
\end{proof}  

{\bf Remark.}  The example $P(x)= \frac{1}{2}(x + x^T)$ on $M_2$ shows the necessity of the 
`completely positive' hypothesis in the part it pertains to.  Note in this example $P$ is positive and contractive, and $I-P$ is completely contractive.

\bigskip

We will `generalize' the bulk of the last result and its proof in Theorem \ref{trivch2}.

We next show that unlike in the $C^*$-algebra case, for projections
on operator algebras
(completely) `bicontractive' is not the same as (completely) `symmetric'.
Also, these both also differ from the notion of (completely) `hermitian'.
We recall that any operator space $X \subset B(H)$ may be unitized to become an operator algebra as follows.
 We define $${\mathcal U}(X) = \biggl\{ \left[ \begin{array}{ccl} \lambda_1 I_H & x\\ 0
&\lambda_2 I_H \end{array} \right] \, :\, x\in X,\ \lambda_1
,\lambda_2 \in\Cdb\biggr\} \subset B(H \oplus H). $$
By definition,
${\mathcal U}(X)$ may be regarded as a subspace of the Paulsen system
 (see  1.3.14 and 2.2.10 in \cite{BLM}, or \cite{Paul}).  It follows from Paulsen's lemma (see \cite[Lemma 8.1]{Paul} or \cite[Lemma 1.3.15]{BLM}) that  if $v \colon X\to X$ is a
linear contraction (resp.\ complete contraction), then the mapping
$\theta_v$ on ${\mathcal U}(X)$  defined by
$$
\theta_v \left( \left[ \begin{array}{ccl} \lambda_1  & x\\ 0
&\lambda_2
\end{array}
\right] \right) \; = \; \left[ \begin{array}{ccl} \lambda_1  \; &
v(x)\\ 0 \; &\lambda_2
\end{array} \right], \qquad x\in X,\ \lambda_1,\lambda_2\in\Cdb,$$
is a contractive (resp.\ completely contractive) homomorphism.

\begin{proposition} \label{one}   Suppose that $X$ is an operator space,
and that $P : X \to X$ is a linear idempotent map.
Then $P$ is completely contractive
(resp.\ completely bicontractive, completely symmetric, completely
hermitian) 
iff the induced map $\tilde{P} : {\mathcal U}(X) 
 \to {\mathcal U}(X)$ is a real completely positive
and completely contractive
(resp.\ completely bicontractive, completely symmetric,  completely hermitian) projection.
In particular these hold (and with the 
word `completely' removed everywhere if one wishes),
if $P : X \to X$ is a linear idempotent map
on a Banach space, when we give $X$ its minimal or maximal operator space
structure (see e.g.\ {\rm \cite{BLM,Paul}}). 
\end{proposition}

\begin{proof}   If $P$ is completely contractive then by Paulsen's lemma
referred to above,  
the unique $*$-linear unital extension of $\tilde{P}$ to ${\mathcal U}(X) + 
{\mathcal U}(X)^*$ is completely contractive and completely positive.
Thus by \cite[Section 2]{BBS}, $\tilde{P}$ is real completely  positive. 
Clearly
$\tilde{P}$ is a projection.  Conversely, if $\tilde{P}$ is completely contractive
then so is $P$.   We note that 
$I- \tilde{P}$ annihilates the diagonal, and acts as
$I-P$ in the $1$-$2$-entry.  Thus  $I-P$ is completely contractive iff
 $I- \tilde{P}$ is completely contractive.
Also note that $2 \tilde{P} - I$ does nothing to the diagonal, and acts as
$2P - I$ in the $2$-$1$-entry; that is $2 \tilde{P} - I
= \widetilde{2P-I}$.  By Paulsen's lemma again, as above,  
$2 \tilde{P} - I$ is completely contractive iff $2P-I$ is completely contractive.
Finally, $I - \tilde{P} + e^{it} \tilde{P}$ multiplies each of the two 
diagonal entries by $e^{it}$, and acts as
$I - P + e^{it} P$ in the $1$-$2$-entry.   Multiplying
by $e^{-it}$, we see again by  Paulsen's lemma 
that this is completely contractive iff $P$ is completely hermitian.

For the  last assertion, if $P$ is contractive (resp.\ bicontractive, symmetric, 
hermitian) then it is completely contractive (resp.\ bicontractive, symmetric, 
hermitian) 
on $X$ with its minimal or maximal operator space
structure by e.g.\  (1.10) and (1.12) in \cite{BLM}.    We may then 
apply the case in the previous paragraph to obtain the result.  Conversely,
if $\tilde{P}$ is contractive (resp.\ bicontractive, etc)., then so is $P$ since $P$ 
is the $1$-$2$ corner of $\tilde{P}$.   The rest is clear.
\end{proof}

The previous result provides examples of real completely positive unital
completely bicontractive (resp.\ completely contractive, completely symmetric)
projections which are not
completely symmetric (resp.\ completely bicontractive, completely hermitian).
As we said, this is in contrast to the $C^*$-algebra case where
(complete) bicontractivity  is equivalent to being (completely)
symmetric \cite{St,FRAD}.  

\begin{example}   A more specific example of a real completely positive
 unital
completely bicontractive projection on an operator algebra
which is not symmetric arises by the last construction, from the following explicit 
completely bicontractive projection which is not completely symmetric.
Let $Y$ be $\Rdb^2$, the latter viewed as a real Banach space whose unit ball is
the unit ball of $\ell^\infty_2$ in the first quadrant and 
the unit ball of $\ell^1_2$ in the second quadrant.
Let $X$ be  the standard  complexification of $Y$.
 Take  $P : X \to X$ to be the usual complexification
of the `projection onto the first coordinate' 
on $Y$.  We thank 
Asvald and Vegard Lima for this example, which is a 
bicontractive projection on a Banach space which is not symmetric. 
Giving $X$ its minimal or maximal operator space
structure makes $X$ an operator space, and makes $P$   (by e.g.\  (1.10) or (1.12) in \cite{BLM})
a completely bicontractive projection which is not symmetric.
Hence by Proposition \ref{one} we get a real completely positive  unital
completely bicontractive projection on an operator algebra
which is not symmetric.  \end{example}

Example \ref{exoo} shows that one cannot extend  completely symmetric real completely 
positive projections to a bicontractive projection on a containing $C^*$-algebra.
Things are much better if $P$ is a unital map:

\begin{lemma} \label{triv}  Let $A$ be a unital operator algebra,
and $P : A \to A$ a completely symmetric unital projection.   
Then $P$ is real completely positive, and 
the range of $P$  is a 
subalgebra of $A$.   Moreover, $P$ extends to a 
completely symmetric unital projection on $C^*_e(A)$ (or on 
$I(A)$).  
\end{lemma}

\begin{proof}   Any completely contractive unital map on $A$ is 
real completely positive as we said in the introduction.
Let $\theta = 2P-I$, so that $P = \frac{1}{2} (I + \theta)$.
By Lemma \ref{sostate} $\theta$ is  a unital complete isometry.  So in fact
$\theta$ is a homomorphism, by the Banach--Stone theorem
for operator algebras \cite{BLM}.  This implies as we said in Lemma \ref{sostate},  that
Ran$(P)$ is a subalgebra of $A$.
We can extend $\theta$ uniquely to a
unital $*$-isomorphism $\pi : C^*_e(A) \to C^*_e(A)$, with
$\pi \circ \pi = I$, and it follows that 
$\tilde{P} = \frac{1}{2} (I + \pi)$ is a 
completely symmetric extension of $P$.   It is
also completely positive.  

Similarly, since one may extend $\pi$ further to a unital 
$*$-automorphism of $I(A)$, there is a completely symmetric 
unital projection on $I(A)$ extending $P$.   \end{proof}

\begin{lemma} \label{trivch}  Let $A$ be a unital operator algebra,
and $P : A \to A$ a completely contractive real positive projection which is bicontractive. 
Then 
 $P(1) = q$ is a projection (not necessarily central), and all the conclusions of 
Theorem {\rm \ref{tr}} and Proposition {\rm \ref{tr2}} hold.   Also, 
 there exists a
unital  completely bicontractive (real completely positive) projection $P' : qAq \to qAq$ such that $P$ is
the  zero map on $q^\perp A + A q^\perp$, and $P =  P'$ on $qAq$.     We have ${\rm Ran}(P) = {\rm Ran}(P')$
\end{lemma}

\begin{proof}   Suppose that $q = P(1)$.  Then 
$q \geq 0$ by Lemma \ref{ispos}, so that the closed algebra $B$ generated by
$q$ and $1$ is a $C^*$-algebra.  Note that 
$P(q^n) = q$ by  Theorem \ref{tr},
 so that $P(B) \subset B$.  Let $Q = P_{\vert B}$, 
this is a  bicontractive projection
on $B$, and it is positive by the proof of Lemma \ref{ispos}. 
  By \cite[Theorem 2]{FRAD} and its proof, $q = P(1)$ is a
partial isometry in $B$, hence is a projection.    Hence all the conclusions of 
Proposition \ref{tr2} and  Theorem {\rm \ref{tr}}  hold.    So  $P(a) = q P(a)q  = P(qaq)$ for $a \in A,$ and $P(A) = q P(A) q = P(qAq)$, 
and 
$P$ `splits' as the sum of  the zero map on $q^\perp A + A q^\perp$, and
a unital  projection $E$ on $qAq$.  \end{proof}

{\bf Remark.} A direct proof that $P(1)$ is a projection in the 
case of the lemma above:
Let $P(1) = q$.  As we saw in   Theorem \ref{tr}, $P(q^n) = q$. Thus $P(u(q)) = q$,
(we defined  $u(\cdot)$
in the introduction).
Suppose that $c$ is a positive scalar with     
$c(q - u(q))$ of norm one.
Then
$$(I-P)(c(q - u(q)) - u(q)) = c(q - u(q)) - u(q)
- (c(q-q) - q) = (1+c)(q - u(q))$$ 
which has norm $> 1$. 
By the contractivity of $I - P$ we have a contradiction,
unless $q = u(q)$.  So $q = P(1)$ is a projection.

\bigskip

The following is the solution to the symmetric projection problem
in the category of  approximately unital operator algebras.

\begin{theorem} \label{trivch2}  Let $A$ be an approximately unital operator algebra,
and $P : A \to A$ a completely symmetric real completely positive projection.
Then the range of $P$  is an approximately unital 
subalgebra of $A$.   Moreover, $P^{**}(1) = q$ is a projection in the multiplier algebra
$M(A)$ (so is both open and closed).  

Set $D = qAq,$ a hereditary subalgebra of $A$ containing $P(A)$.
There exists a period $2$ surjective completely 
isometric homomorphism $\theta : A \to A$ such that $\theta(q) = q$, so that $\theta$ restricts to 
a  period $2$ surjective completely 
isometric homomorphism $D \to D$.    Also,  $P$ is
the  zero map on $q^\perp A + A q^\perp + q^\perp A q^\perp$, and $$P =  \frac{1}{2} (I + \theta) \; \; \;
{\rm on} \; D.$$
In fact $$P(a)=  \frac{1}{2} (a +  \theta(a) (2q-1)) \, , \qquad a \in A.$$    The range of $P$ is exactly the set of fixed points of $\theta_{\vert D}$  in $D$.

Conversely, any 
map of the form in the last equation, for a period $2$ surjective completely 
isometric homomorphism $\theta : A \to A$ and a projection $q \in M(A)$ with $\theta^{**}(q) = q$, is a 
completely symmetric real completely positive projection. 
\end{theorem}

\begin{proof}    
Applying  Lemma \ref{trivch} to $P^{**}$ , we see that $P^{**}(1) = q$ is a projection, and 
all the conclusions of Proposition \ref{tr2} and  Theorem \ref{tr} are true for us.  
We will silently be using facts from these results below. In particular $q$ is an open projection, so supports
an approximately unital subalgebra $D$ of $A$ with $D^{\perp \perp} = qA^{**}q$ (see \cite{BHN}).   Then  
$\theta = 2 P - I$ is a  linear completely isometric surjection on $A$ by Lemma \ref{sostate}.   So by the Banach--Stone theorem for operator algebras \cite[Theorem 4.5.13]{BLM},
   there exists a completely isometric surjective homomorphism $\pi : A \to A$ and a unitary $u$
 with $u, u^{-1} \in M(A)$ with $\theta = 
\pi(\cdot) u$.    We have $$\theta^{**}(1) = 2 P^{**}(1) - 1 = 
2 q - 1 = \pi^{**}(1) u = u,$$
so that $u$ is a selfadjoint unitary (a symmetry), and $q \in M(A)$.    So $qAq = D \subset A$.
Since $P(A) = qP(A)q$,
the range of $P$ is contained in $D$, 
and the range of $P$ is exactly the set of fixed points of $\theta$, which all lie in $D$.  
   This implies  that
Ran$(P)$ is a subalgebra of $A$.  It is approximately unital and $P$ is
real completely positive by Corollary \ref{acp}.   

We have $\theta^{**}(q) = q$ and  $\pi^{**}(q) = \theta^{**}(q) u = qu = q$.
Then $2P(a) - a = \pi(a)(2q-1)$ and $P(a)=  \frac{1}{2} (a +  \pi(a) (2q-1))$.
Indeed $$\pi(a) = (2P(a) - a) u = (2P(a) - a) (2q-1) = 
2P(a) - 2 a q + a, \qquad a \in A.$$ 
From this, or otherwise, one sees that  $\pi$ equals $\theta$ on $D$, and $\pi(D) = (2P - I)(D) \subset D$.  
However  $D = \theta^2(D) \subset \theta(D) = \pi(D)$, so $\pi(D) = D$.     This completes the main part 
of the theorem.

For the converse, note that such $P(a)=  \frac{1}{2} (a +  \pi(a) (2q-1))$ is clearly completely
symmetric on $A$, and 
$$P(P(a)) = P(\frac{1}{2} (a +  \pi(a) (2q-1))) = \frac{1}{4} (a + 2 \pi(a)  (2q-1) + \pi(\pi(a) (2q-1) )(2q-1) =
P(a),$$  since $\pi$ is period $2$, $\pi^{**}(q) = q$, and $2q-1$ is a symmetry.  
We have $P^{**}(1) = \frac{1}{2} (1 +  \pi^{**}(1) (2q-1)) = q$, so that $P$ is
real completely positive by Proposition \ref{tr2}.
\end{proof}

It follows easily from the last theorem that 
 a completely symmetric real completely positive projection $P$ on $A$
extends to a completely
symmetric projection $\tilde{P}$ on the $C^*$-envelope of $A$.
Moreover, $\tilde{P}(x) =  \frac{1}{2} (a +  \tilde{\pi}(a) (2q-1))$,
for a $*$-automorphism $\tilde{\pi}$.    However, this
extension will not in general be positive.

In the work \cite{BNj} in progress, we prove the Jordan algebra variant of the last result.

\section{The bicontractive projection problem}

The following could be compared with Corollary \ref{newaei}. 

\begin{lemma} \label{yoho}  Let $A$ be a unital operator algebra,
and $P : A \to A$ a  unital projection with $I-P$  completely contractive.  
Let $C = (I-P)(A) = {\rm Ker}(P)$.  Then $C^2 \subset P(A)$.  We also have
\begin{itemize} \item [(1)]  $C$ is a subalgebra of $A$ iff $C^{2} = (0)$. 
 \item [(2)]  If $P$ is completely bicontractive (or more generally, if $P(aP(b)) = P(P(a) P(b))$ for $a, b \in A$)
then $P(A) C + C P(A) \subset C$, and $C^3 \subset P(A)$.  
 \item [(3)]    If the conditions in 
{\rm (1)}  hold, and if 
$P$ is completely contractive (or more generally, if $P(aP(b)) = P(P(a) P(b))$ for $a, b \in A$)
then $C$ is actually an ideal in $A$.   
\item [(4)]    If $P(A) C + C P(A) \subset C$
(see 
{\rm (2)}), then  $\theta = 2P- I_A$ is a homomorphism
iff $P(A)$ is a subalgebra of $A$, and then  
the range of $P$ is the set of fixed points of this automorphism $\theta$. \end{itemize}  
\end{lemma}

\begin{proof}    Of course $A = C + B$, where $B = P(A)$.  By Youngson's result \cite{Y}  applied to 
an extension $Q$ of $I-P$ to a completely contractive projection on $I(A)$ (which exists by an easier variant 
of the proof in Theorem \ref{tr}), we have
$(I-P)(w z) = Q(w (I-P)(1)^* z) = 0$ for $z, w \in C$.   So $P(wz) = wz \in B$,
and this is zero if the kernel  is a subalgebra.    In any case,
$C^2 \subset B$.   Assuming that $P$ is completely contractive (or that $P(aP(b)) = P(P(a) P(b))$), 
if $z \in B$ and $w \in C$, then $P(wz) = P(P(w)z) = 0$, so $w z \in C$.  So $C B \subset C$ and similarly
$B C \subset C$.   Thus $C^3 \subset B C \subset C$.  Hence if also $C$ is
a subalgebra, then it is an ideal.    

For (4),    we 
may decompose  $A = C \oplus B$, where $1_A \in 
B = P(A)$,
and we have the relations $C^2 \subset B, C B + B C \subset C$.   Using the latter it is a simple computation that 
the period 2 map $\theta : x + y \mapsto x-y$ for $x \in B, y \in C$ is a homomorphism (indeed an automorphism)
on $A$ iff $P(A)$ is a subalgebra of $A$; and clearly $P(A)$ is the  fixed points of $\theta$.  
\end{proof} 

{\bf Remarks.}  1) \ Replacing $P$ by $I-P$ we see that if $P : A \to A$ is a completely contractive projection with $P(1) = 0$,
then $P(A)$ is a subalgebra iff $P(A)^2 = (0)$.    

\medskip

2) \ The kernel of a bicontractive projection need not be a subalgebra.  For example, consider $P(f)(x) = \frac{1}{2}(f(x) + f(-x))$ for
$f \in C([-1,1])$.  

\bigskip

The following clarifies the unital version of the 
`bicontractive projection problem' in relation to the existence of an associated period two automorphism.

\begin{corollary} \label{san1}  If $P : A \to A$ is a unital idempotent on a unital operator algebra let
$\theta : A \to A$ be the associated linear period 2 automorphism $x + y \mapsto x-y$ for $x \in {\rm Ran}(P), y \in 
{\rm Ker}(P)$.  Then $P$ is
completely bicontractive iff  $\Vert I \pm \theta \Vert_{\rm cb} \leq 2$.   If these hold then the range of $P$ is a subalgebra iff $\theta$ is also a homomorphism, and then
the range of $P$ is the set of fixed points of this automorphism $\theta$.  
Also, $P$ is completely symmetric iff $\theta$ is completely contractive.  
\end{corollary}

\begin{proof}  The first and last assertions are obvious, and the second follows by Lemma \ref{yoho}   (4).    
 \end{proof}

For us, the `bicontractive  projection problem' is
whether the range of a completely bicontractive real
completely positive projection on an approximately unital operator algebra
$A$, is an approximately unital 
 subalgebra of $A$.    This is not obvious, although there are easy counterexamples if
one drops some of the hypotheses.  For example consider the projection $P$ on the upper triangular $2 \times 2$ matrices 
taking the matrix with rows
$[a \; b]$ and $[0 \; c]$ to the matrix with rows
$[(a-c)/2 \;  \; 0]$ and $[0 \;  \; (c-a)/2]$ .  This is completely contractive and extends to 
a completely contractive projection on the containing $C^*$-algebra, and can be shown to be bicontractive
with $P(1) = 0$, and its range is not a subalgebra.

 We now show that the `bicontractive  projection problem'
can be reduced to the case that $A$ is unital and  $P(1) = 1$, by a sequence of three
reductions.   First, if 
$P : A \to A$ is a completely bicontractive real
completely positive projection on an approximately unital operator algebra
$A$, then $P^{**}$ is a completely bicontractive real
completely positive projection on a unital operator algebra.
Thus henceforth in this section in the next we assume that $A$ is unital.
Second, Lemma \ref{trivch}  allows us to reduce further to the case that $P(1) = 1$:
it asserted that $P(1) = q$ is a projection (not necessarily central), all the conclusions of 
 Theorem {\rm \ref{tr}} and Proposition \ref{tr2}  hold, and 
 there exists a
unital  completely bicontractive (real completely positive) projection $P' : qAq \to qAq$ such that $P$ is
the  zero map on $q^\perp A + A q^\perp$, and $P =  P'$ on $qAq$, and ${\rm Ran}(P) = {\rm Ran}(P')$.

Then the third of our reductions of the completely bicontractive  projection problem
puts the problem in a
`standard position'.   
Of course $P(A)$ is a subalgebra if and only if   $Q(D)$ is a subalgebra,
where $D$ is the closed subalgebra of $A$ generated by $P(A)$, and
$Q = P_{\vert D}$, which is a completely bicontractive unital projection on $D$.   That is, we may as well replace $A$ by  the 
closed subalgebra generated by $P(A)$.

Thus by these three steps above,
 we have reduced the  completely bicontractive  projection problem on approximately 
unital operator algebras to  the `standard position' of a unital projection on a unital operator algebra,
whose
range generates $A$.  In this situation we obtain the following structural result.

\begin{corollary}   \label{newa}  Let $P$ be a completely bicontractive unital projection on a unital operator algebra $A$. 
Let $D$ be the algebra generated by $P(A)$.  
 Then $(I-P)(D) = {\rm Ker}(P_{|D})$ is an ideal in $D$ and the product of any two elements in this ideal is zero.  \end{corollary}

\begin{proof}   We saw at the end of Section 2 that $(I-P)(D) = {\rm Ker}(P_{|D})$ is an ideal in $D$, and in the 
notation there it equals  $D \cap e^\perp M e^\perp$.   Since $D \cap e^\perp M e^\perp$
is a subalgebra of $D$
the result follows from Lemma \ref{yoho}.
\end{proof}   

{\bf Remark.}  Note that $(I-P)(D) \subset e^\perp A$,  but $P(D)$ is not a subset of $eA$.

\bigskip

The above shows that we can also solve the bicontractive  projection
problem in the affirmative for real completely 
positive completely bicontractive projections $P$ on a unital operator algebra $A$ 
such that the closed algebra generated by $A$ is semiprime (that is, it has no nontrivial
square-zero ideals): 

\begin{corollary} \label{m4b}   Let $P : A \to A$ be a  real completely positive 
completely bicontractive projection on a unital 
operator algebra.
If $A$  is an operator algebra containing no nonzero nilpotents, then $P(A)$ is a subalgebra of $A$.    
Also if the  closed algebra $D$ generated by $P(A)$  is 
 semiprime, then  $P(A)$ is a subalgebra of $A$.
\end{corollary} 

\begin{proof}  By the  second reduction above, we may assume that $P$ is
unital.   By Corollary   \ref{newa}, $(I-P)(D)$ is an ideal in $D$ with square zero,  and so  is $(0)$ in these cases.  So $P(A) = P(D) = D$, a subalgebra.  \end{proof} 

For the subcategory of uniform
algebras (that is, closed 
subalgebras of $C(K)$, for compact $K$) which are unital or approximately unital, the
bicontractive  projection problem coincides with the symmetric projection problem,
and again there is a complete solution:

\begin{theorem}  \label{rcpua}  Let $P : A \to A$ be a  real  positive 
 bicontractive projection on a uniform algebra $A$, and suppose that $A$ is unital or approximately 
unital).  Then  $P$ is (completely) symmetric,
and so  we have all the conclusions of Theorem {\rm   \ref{trivch2}}.  In particular, 
 $P(A)$ is a subalgebra of $A$,
and $P$ is a conditional expectation.    \end{theorem}

\begin{proof}   Here bicontractive projections are the same as completely bicontractive 
projections  (by e.g.\  (1.10) in \cite{BLM}).  By the obvious variant of the usual proof that 
positive maps into a $C(K)$ space are  completely positive, 
we have that real positive maps into a uniform algebra are real 
completely positive.  
By the first two reductions described above we can assume that 
$A$ and $P$ are unital.   We also know that $B = P(A)$ is a subalgebra by Corollary \ref{m4b}, since
e.g.\ nonzero nilpotents cannot exist in a function algebra. 
Thus by Corollary \ref{san1} the map $\theta(x+ y) = x-y$ described there is an algebra automorphism of $A$,
hence a (completely) isometric isomorphism (since norm equals spectral radius).  
 So  $P = \frac{1}{2} (I + \theta)$
is  (completely) symmetric, and Theorem {\rm   \ref{trivch2}} applies.   
\end{proof}  

{\bf Remark.}   The  idea in the last proof that $\theta$ is automatically isometric, since it is
an algebraic automorphism of a uniform algebra, and that this implies that 
$P$ is symmetric, was found together with Joel Feinstein after submission of our original paper.

\bigskip

By e.g.\ Corollary \ref{san1}, to find a counterexample to the conjecture that all (completely real positive) completely bicontractive unital projections have range 
which is a subalgebra, 
we need a unital operator algebra $D = C \oplus E$ where $C \neq (0), C^2 = 0, CE + EC \subset C,$ and 
$1_D \in E$ and $E$ generates $D$, so in particular $E^2$ is not a subset of $E$, and with  the projection maps onto $C$ and $E$ completely contractive.     This is easy in a Banach algebra, one may equip  $\ell^1_3$, with the standard basis identified with symbols 
$1, a, b$ satisfying relations like $b^2 = 0, a^2 = b$, etc.\ (setting $C = \{ b \}, E = {\rm Span} \{1,a \}$).   To find an operator algebra example, we make a general construction.  

\begin{example} \label{Gco}  Let $V$ be a closed subspace of $B(H) \oplus B(H)$, viewed as elements of $B(H^{(2)})$ supported on the $1$-$1$ and $2$-$2$ entries.
Write $v_1$ and $v_2$ for the two `parts' of an element $v \in V$.   Let $C$ be the closed span of
 $\{ v_1 w_2 : v, w \in V \}$, and we will assume that $C \neq (0)$.
Let $B$ be the set of elements of  $B(H^{(3)})$ of form 
$$\left[ \begin{array}{ccl} \lambda I & v_1 & c \\ 0
& \lambda I & v_2 \\ 0 & 0 & \lambda I \end{array} \right]   \; , \qquad \lambda \in \Cdb, v \in V, c \in C.$$
Then the copy of $C$ is an ideal in $B$ with square zero, and it is generated by the copy of $V$ in $B$.  If $E$ is the sum of this
copy of $V$ and $\Cdb I_{H^{(3)}}$, then all the conditions in the last paragraph  needed for a counterexample hold, with the 
exception of the projection onto $E$ being completely contractive.    We remark that one may also describe $B$ more abstractly 
as a set $B = \Cdb 1 + V + C$ in $B(K)$ where $V$ and $C$ are closed subspaces of $B(K)$ with the properties that $(0) \neq C = V^2,  V^3= C^2 = (0)$ plus one more condition ensuring that $p_1 V (1-p) = (0)$ where $p_1$ is the left support projection of $C$ and $p$ is the left support projection of $C + V$.
 
To fix the exception noted in the last paragraph, we consider the subalgebra
 $A = A(V)$  of $B(H^{(7)})$  consisting of all elements  of form 
$$\left[ \begin{array}{ccccccl} \lambda I & v_1 & c & 0 & 0 & 0 & 0 \\ 0
& \lambda I & v_2  & 0 & 0 & 0  & 0 \\ 0 & 0 & \lambda I  & 0 & 0 & 0 & 0 \\
0  & 0 & 0 & \lambda I & 0 & 2 v_1 & 0 
 \\
0 & 0  & 0 & 0 & \lambda I & 0 & 2 v_2  \\
0 & 0 & 0  & 0 & 0 & \lambda I & 0   \\
0 & 0 & 0 & 0  & 0 & 0 &  \lambda I 
\end{array} \right]   \; , \qquad \lambda \in \Cdb, v \in V, c \in C.$$    The square-zero ideal in $A$ consisting of the copy of $C$  we will abusively write again as $C$,
and again let $E$ be the sum of $\Cdb I_{H^{(7)}}$ and the isomorphic 
copy of $V$ in $A(V)$, so that 
 $A = C \oplus E$.   Here $C \neq (0), C^2 = 0, CE + EC \subset C,$ and 
$1_D \in E$ and $E$ generates $A$, 
as desired.  The canonical  
projection map from $A$ onto $C$ is obviously  completely contractive.    

A particularly simple case is when  $H$ is one dimensional, so that $A \subset M_7$, and where $V = \Cdb I_2$.  This algebra is 
obviously essentially just (i.e.\ is completely isometrically isomorphic to) the subalgebra of $M_5$ of matrices
$$\left[ \begin{array}{ccccl} \lambda & \nu & c & 0 & 0 \\ 0
& \lambda & \nu  & 0 & 0 \\ 0 & 0 & \lambda  & 0 & 0  \\
0  & 0 & 0 & \lambda & 2 \nu 
 \\
0 & 0 & 0  & 0 &  \lambda   
\end{array} \right]   \; , \qquad \lambda , \nu, c \in \Cdb,$$ with the projection 
being `replacing the $1$-$3$ entry' by $0$.   
Thus $5 \times 5$ matrices of scalars  will suffice to give an interesting example.  However we will need the more general construction later
to produce a more specialized counterexample.  
\end{example}

Consider the algebra ${\mathcal U}(V)$ constructed from $V$ as described in the paragraph 
above 
Proposition \ref{one}.
We define ${\mathcal U}_0(V)$ to be the subalgebra of ${\mathcal U}(V)$ consisting 
of elements of ${\mathcal U}(V)$  with  the two `diagonal entries' identical.   It is a subalgebra of of $B(H^{(4)})$.

\begin{lemma} \label{iscov}     In the situation
of Example {\rm \ref{Gco}}, the map  $j : {\mathcal U}_0(V) \to B(H^{(3)})$ given by 
$$\left[ \begin{array}{ccl} \lambda I & v \\ 0
&\lambda I \end{array} \right]   \mapsto \left[ \begin{array}{ccl} \lambda I & \frac{1}{2} v_1 & 0 \\ 0
& \lambda I & \frac{1}{2} v_2 \\ 0 & 0 & \lambda I \end{array} \right] 
$$
is completely contractive and unital.   \end{lemma}  

\begin{proof}    
To prove that  it is 
  completely contractive simply notice that $j$ may be viewed as the composition of the canonical map 
${\mathcal U}(V) \to {\mathcal U}(W)$ where $W$ is the copy of $V$ in the algebra $B$ above, 
and  the map $M_2(B) \to B$ 
given by pre- and postmultiplying by 
$[\frac{1}{\sqrt{2}} I_{H^{(3)}} \; \; \frac{1}{\sqrt{2}} I_{H^{(3)}}]$ and its transpose.
\end{proof}  

\begin{corollary}   \label{gex}    If $A = A(V)$ is the unital operator algebra above in $B(H^{(7)})$ then the canonical projection $P : A \to A$ which
replaces the $1$-$3$ entry by $0$, whose kernel is a nontrivial square-zero ideal
$C$ generated by $P(A)$, is  a (real completely positive)
 completely bicontractive and unital projection, but its range  is not a subalgebra, and it need 
not even be a Jordan subalgebra.  
A particularly simple `case' is the algebra (completely isometrically isomorphic to the algebra
of) of $5 \times 5$ scalar matrices described above Lemma {\rm \ref{iscov}}.
  \end{corollary}  

\begin{proof}     In the $5 \times 5$ matrix example it is easily checked that $P(A)$
is not closed under squares, hence is not a Jordan subalgebra.
  It remains to prove that $P$ is completely contractive.   However $P$ is the composition of 
the canonical map $B(H^{(3)} \oplus H^{(4)}) \to B(H^{(4)})$ restricted to $A$,  and the 
map $x \mapsto j(x) \oplus x$ on ${\mathcal U}_0(V)$, where $j$  is as in Lemma \ref{iscov}. 
\end{proof}

The following is another rather general condition under which the 
 completely bicontractive projection problem 
 is soluble.  Indeed as we said in the introduction, all 
examples known to us of real completely positive completely bicontractive projections
on unital operator algebras, whose range is a subalgebra, do
satisfy the criterion in Theorem \ref{mac}.

\begin{theorem} \label{mac}   Let $A$ be a unital operator algebra in a von Neumann algebra $M$ 
(which could 
be taken to be $B(H)$, or $I(A)^{**}$ as above)
and let $P : A \to A$ be a unital completely bicontractive projection.
Let $D$ be the closed algebra generated by $P(A)$, and let $C = (I-P)(D)$.

Suppose further that either 
 $C \; P(A)^* \subset \overline{MC}^{w*}$, the weak* 
closed  left ideal
in $M$ generated by $C$ (this is equivalent to saying that the 
left support projection of $P(A) C^*)$ is dominated by the
right support projection of $C$).
Alternatively, assume that $P(A)^* C$ is contained in the
weak*
closed  right ideal
in $M$ generated by $C$ (or equivalently that
the right support projection of $C^* P(A)$ 
 is dominated by the left support projection of $C$). 
 Then $P(A)$ is a subalgebra of $A$.
\end{theorem}  

\begin{proof}  
 We assume the `left ideal' condition, the other case is similar and left to
the reader (or can be seen by looking at the opposite algebra' $A^{{\rm op}}$).
By replacing $A$ by $D$ we may assume that $P(A)$ generates $A$.
Suppose that $M$ is a von Neumann algebra on a Hilbert space $H$.
We set $p_1 =  \vee_{z \in C} \,  r(z) r(z)^*$ to be the left support projection 
of $C$, and set $p_2 = \vee_{z \in C} \, r(z)^* r(z)$, the right support projection of $C$.    Note that 
$z p_1 = 0$ for $z \in C$, which implies that $p_2  p_1 = 0$.   Let $p = p_1 + p_2$, a projection.  Our hypothesis
is equivalent to saying that  $$C \; P(A)^* = C \; P(A)^* p_2.$$ 
We may write
$$M = (1-p)M(1-p) + (1-p)M p +p M (1-p) + pMp .$$
Thus $M$ may be pictured as the direct sum of a von Neumann algebra with four corners (thus having a $2 \times 2$ matrix form).   
Let  $y \in P(A)$.   Then $y C \subset C \subset p M$, so that $(1-p) y p_1 = 0$.    On the other hand, by hypothesis,
$p_2 y^* (1-p) = p_2 y^* p_2  (1-p) = 0$.  Thus  $(1-p) y p_2 = 0$ and so  $(1-p) y p= 0$.  Therefore
$$y =  (1-p)y(1-p) +p y (1-p) + pyp  =    y(1-p)  + pyp .$$
Furthermore $p_2 y C \subset p_2 C = p_2 p_1 C = (0)$, and so $p_2 y p_1 = 0$.   And by hypothesis,
$p_1 y p_2 = p_1 (p_2 y^*)^* = p_1 (p_2 y^* p_2)^* = 0$.    So 
$$p_1 P(A) p_2 = (0).$$   Thus
$$y =  y(1-p) + p_1 yp_1 +  p_2 yp_2 .$$  
If also $x \in P(A)$ then $x =  x(1-p) + p_1 xp_1 +  p_2 x p_2 $, and so 
$$x y = x (1-p) y(1-p) + p_1 x p_1 y p_1 + p_2 x p_2 y p_2 ,$$
and so again $p_1 x y p_2 = 0$.   Since $C = (I-P)(A)$ we see that 
$$(I-P)(xy) = p_1 (I-P)(xy) p_2 =   p_1xy p_2 - p_1 P(xy) p_2 = 0,$$
so that $xy = P(xy) \in P(A)$.  So $P(A)$ is a subalgebra.  
 \end{proof}

{\bf Remarks.  } \ 1) \  If  $C \; P(A)^* \subset [BC]$ where $B = C^*_{\rm e}(A)$ then the 
first hypothesis in the previous result holds.

\medskip

2) \ If  $A$ is  the counterexample algebra
of $5 \times 5$ scalar matrices described above Lemma \ref{iscov}, it is very illustrative to compute 
the various associated objects of interest in our paper.    We leave the details to the reader as an exercise.
Here $C^*(A) = C^*_{\rm e}(A) = 
I(A) = M_3 \oplus M_2 \subset M_5$.   If $P$ is the projection in that 
example, namely the map that replaces the $1$-$3$ entry with $0$, then $C^*(P(A)) = I(P(A)) = 0 \oplus M_2 \subset M_5$.
A completely contractive completely positive projection $\tilde{P}$ on $C^*(A) = I(A)$ that extends $P$ is 
the map 
$$x \oplus \left[ \begin{array}{cccl} a & b    \\ c & d \end{array} \right]   \; \mapsto
\left[ \begin{array}{ccccl} \frac{a+d}{2} &   \frac{b}{2} & 0 & 0 & 0 \\  \frac{c}{2}
&  \frac{a+d}{2} &   \frac{b}{2}  & 0 & 0 \\ 0 &  \frac{c}{2} &  \frac{a+d}{2} & 0 & 0  \\
0  & 0 & 0 & a &  b
 \\
0 & 0 & 0  &  c &  d   
\end{array} \right]   \; , \qquad x \in M_3.$$ 
To see that this is completely contractive is a tiny modification of the proof of Lemma \ref{iscov}.   A  
completely contractive projection extending $I-P$ is is the projection onto the $1$-$3$ coordinate.   The support projection of $\tilde{P}$ defined just before Proposition \ref{ES} 
   is $e = 0 \oplus I_2$, which has complement
$r = I_3 \oplus 0$.  The projections $p_1, p_2, p$ from the 
proof of Theorem \ref{mac} are  $p_1 = E_{11}, p_2 = E_{33}, p = E_{11} + E_{33}$; and $1-p
= E_{22} + E_{44} + E_{55}$.  Also, $r - p = E_{22}$.  
Note that $C \; P(A)^* p_2 \neq C \; P(A)^*$, of course.   Indeed this example is
an excellent illustration of what is going on in the proof of
Theorem \ref{mac}.    Note that if we change the 
definition of $A$ by replacing either the $2$-$3$ entry or the $3$-$2$ entry
then the hypotheses of Theorem \ref{mac} {\em are} satisfied.  

Examining why the general example described in \ref{Gco} does not
satisfy the hypotheses of Theorem \ref{mac} is illustrative:
it is not hard to see that if it did then $v_1 w_2 z_2^* = 0$
for all $v, w, z \in V$.   However if 
$0 \neq \sum_{k=1}^n v^k_1 w^k_2 \in C$  then we 
obtain the contradiction
$0 \neq (\sum_{k=1}^n v^k_1 w^k_2)(\sum_{k=1}^n v^k_1 w^k_2)^* = 0$.

\bigskip

It would be interesting to investigate other conditions that might imply that 
$P(A)$ is a subalgebra, particularly when in the `standard position' 
(namely  $P : A \to A$ is a unital completely  bicontractive  projection whose range generates 
$A$ as an operator algebra).   Some which might be worth
investigating are 
if the algebra $C$ in Theorem \ref{mac} is a maximal ideal in $A$, or if $C$ contains the radical of $A$. 
Note that any one of these conditions rules out our counterexamples above.  

\section{Another condition}

We now look at another condition on a 
completely contractive projection $P$ which is automatic 
for bicontractive projections in the $C^*$-algebra case,
namely that the induced
projection on ${\rm Re}(A)$ is bicontractive.  We will not assume that
the induced
projection on ${\rm Re}(A)$ has {\em completely} contractive 
complementary projection $I-P$.  We are not able to solve the problem yet,
but have made some progress towards the solution.

\begin{lemma} \label{m1}  Let $A$ be a unital operator algebra,
and let $P : A \to A$ a unital completely contractive projection such that the induced 
projection on ${\rm Re}(A)$ is bicontractive.  We also write 
$P$ for an extension to a  unital completely contractive weak* continuous projection on the 
von Neumann algebra $B^{**}$, where $B$ is a $C^*$-algebra containing $A$ as a unital-subalgebra (see the 
argument in the proof of  Corollary   {\rm \ref{newaei}}).
   Let $e$ be the support projection of $P$ on $B$ as in {\rm \cite[p.\ 129]{ES}}.
If $x \in A \cap  e^\perp B e^\perp$ and $x = a+ib$ with $a = a^*, b = b^*$, then 
$\Vert a_+ \Vert = \Vert a_- \Vert = \Vert b_+ \Vert = \Vert b_- \Vert $. \end{lemma}

\begin{proof}     Suppose that $\Vert a \Vert \leq 1$.  By the Kadison-Schwarz inequality
$$P(a)^* P(a) \leq P(a^* a) = P(e^\perp  a^* a e^\perp) \leq P(e^\perp) = 0.$$
So $P(a) = 0$, and similarly $P(b) = 0$.   Suppose that 
$\Vert a_+ \Vert > \Vert a_- \Vert$.    Then by the spectral theorem for $a$, 
there exists $\epsilon > 0$ with $$\Vert a - \epsilon 1
\Vert = \Vert a_+ - a_-  - \epsilon 1
\Vert < \Vert a_+ - a_- \Vert = \Vert a 
\Vert.$$     Thus 
$\Vert a - \epsilon 1
\Vert < \Vert (I-P)(a - \epsilon 1) \Vert$, a contradiction.   So $\Vert a_+ \Vert \leq \Vert a_- \Vert$.
A similar argument shows that $\Vert a_- \Vert \leq \Vert a_+\Vert$, so that 
$\Vert a_+ \Vert = \Vert a_- \Vert$.  Similarly (or by replacing $x$ by $ix$),
$\Vert b_+ \Vert = \Vert b_- \Vert$.

Let $y = {\rm Re}(2x - x^2)$.   The last paragraph shows that $\Vert y \Vert = \Vert y_+ \Vert = \Vert y_- \Vert$.
Now assume that $\Vert a \Vert = 1 \geq \Vert b \Vert$.   So $\Vert a_\pm \Vert = 1$.  Let $\psi$ be a state 
with $\psi(a_-) =1$.   Then  $\psi(a_+) = 0$ or else $\psi(|a|) = \psi(a_+) + \psi(a_-) > 1$ which is impossible.  
 By standard arguments these imply that 
$\psi(a_-^2) =  1$ and  $\psi(a_+^2) =  0$.  Since $y = 2 a_+ - 2 a_- - (a_+^2 + a_-^2) + (b_+^2 + b_-^2),$ we have $$\psi(y) = -3  + \psi(b_+^2 + b_-^2).$$
It is well known that for a selfadjoint operator $T = T_+ - T_- = R-S$ with $R, S \geq 0$, we have
$\Vert T_+ \Vert \leq \Vert R \Vert$.  Thus $$\Vert y \Vert = \Vert y_+ \Vert \leq \Vert 2a_+ - a_+^2 +(b_+^2 + b_-^2)
\Vert \leq 2.$$   Since  $\psi(y) = -3  + \psi(b_+^2 + b_-^2)$ we must have $\psi(b_+^2 + b_-^2) = 1$,
so that $\Vert b^2 \Vert = 1 = \Vert b \Vert$.   Replacing $x$ by $ix$, we see that  $\Vert a \Vert = \Vert b \Vert$. 
  \end{proof}

\begin{lemma} \label{m2}  Let $A$ be a unital operator algebra,
and let $P : A \to A$ a unital completely contractive projection such that the induced 
projection on ${\rm Re}(A)$ is bicontractive.  We also write 
$P$ for an extension to a  unital completely contractive weak* continuous projection on the 
von Neumann algebra $B^{**}$, where $B$ is a $C^*$-algebra  containing $A$ as a unital-subalgebra (as in the
last result).
  Let $e$ be the support projection of $P$ on $B$ as in {\rm \cite[p.\ 129]{ES}}.
If $x \in A \cap  e^\perp B e^\perp$ and $x = a+ib$ with $a = a^*, b = b^*$, and 
$\Vert a \Vert = 1$, then $u(a)^2 = u(b)^2$. \end{lemma}

\begin{proof}    Since $b = b^*$, $u(b)$ is a selfadjoint tripotent 
and $u(b)^2$ is a projection.    It is well known that $u(x)^* u(x) = u(x^* x)$ (to see this
note that $x (x^*x)^n \to x u(x^*x)$, so that $x u(x^*x) = u(x)$ from which the 
relation is easy).  Hence  $u(b)^2 = u(b^2) = u(b_+^2 + b_-^2)$.
As we saw in the last proof, if  $\psi$ is a state 
with $\psi(a_-) =1$ then $\psi(b_+^2 + b_-^2) = 1$.     Now 
 $\psi(a_-) =1$ iff  $\psi(u(a_-)) =1$, and  $\psi(b_+^2 + b_-^2) = 1$
iff $\psi(u(b)^2) = 1$ (see \cite[Lemma 3.3 (i)]{ER4}).  
So $\{ u(a_-) \}' \cap S(B) \subset \{ u(b)^2 \}' \cap S(B)$, where $S(B)$ is the state space
and the `prime' is as in \cite{ER4}.   From this, as is well known
(and simple to prove), we have that $u(a_-) \leq  u(b)^2$.   Similarly,
$u(a_+) \leq  u(b)^2$, so that $u(a_-) + u(a_+) = u(a)^2 \leq  u(b)^2$.  
Similarly, $u(b)^2 \leq  u(a)^2$, so $u(a)^2 = u(b)^2$.  \end{proof}

\begin{lemma} \label{m3}  Let $A$ be a unital operator algebra,
and let $P : A \to A$ a unital completely contractive projection such that the induced 
projection on ${\rm Re}(A)$ is  bicontractive.  We also write 
$P$ for an extension to a  unital completely contractive weak* continuous projection on the 
von Neumann algebra $B^{**}$, where $B$ is a $C^*$-algebra containing $A$ as a unital-subalgebra
 (as in the last results).
   Let $e$ be the support projection of $P$ on $B$ as in {\rm \cite[p.\ 129]{ES}}.
Suppose that $x \in A \cap  e^\perp B e^\perp$ has norm $1$.   Then $u(x)^2 = 0$ and 
$x = u(x) + y$ for some $y \in B^{**}$
with $u(x) y = u(x) y^* = y u(x) = y^* u(x) =  0$.   Finally, 
$\Vert {\rm Re} \, x \Vert = 1/2$. \end{lemma}

\begin{proof}     Suppose that $x \in A \cap  e^\perp B e^\perp$, and choose an angle $\theta$ so that $\Vert {\rm Re} \, (e^{i \theta} x) \Vert$ is maximized.
By Lemma \ref{m1} this equals   $\Vert {\rm Im} \, (e^{i \theta} x) \Vert$.   Write 
$z = e^{i \theta} x = a+ib$ with $a = a^*, b = b^*$.  
Scale $z$ so that $\Vert a \Vert = 1$ (so $\Vert b \Vert = 1$ by Lemma \ref{m1}).  Write $a = u(a) + a_\perp$ and $b = u(b) + b_\perp$.
Note that $u(a)  a_\perp
= u(a) (a - u(a)) = 0,$ since $u(a) a = u(a)^3 a = u(a)^2$.
Similarly $a_\perp u(a) = 0$, and $b_\perp u(b) = u(b) b_\perp = 0$.
Since  $u(b)^2 = u(a)^2$ by Lemma \ref{m2}, we have $u(a) b_\perp = 
u(a) u(b)^2 b_\perp = 0$, and similarly $b_\perp u(a) = a_\perp u(b) = u(b) a_\perp = 0$.
Hence  $a_\perp$ and $b_\perp$ are contractions by the orthogonality of 
$u(a)$ and $a_\perp$, and of $u(b)$ and $b_\perp$.  
Consider $$\frac{1+i}{\sqrt{2}} \; (a+ib) = \frac{a-b}{\sqrt{2}} + i \;  \frac{a+b}{\sqrt{2}} .$$
By the maximality property of $\theta$ we must have $\Vert  \frac{a-b}{\sqrt{2}} \Vert = \Vert
 \frac{a+b}{\sqrt{2}} \Vert \leq 1$. 
Now $$\frac{u(a) - u(b)}{\sqrt{2}}  = u(a)^2  \; \frac{a-b}{\sqrt{2}} \; \; \; \text{and} 
\; \; \; \frac{u(a) + u(b)}{\sqrt{2}}  = u(a)^2 \; 
  \frac{a+b}{\sqrt{2}},$$ so these are contractions.   Squaring each of these we see that 
$u(a)^2 - \frac{1}{2}(u(a) u(b) + u(b) u(a))$ and $u(a)^2 + \frac{1}{2}(u(a) u(b) + u(b) u(a))$
are contractions.   Since $u(a)^2$ is a projection, hence an extreme point, we deduce that 
$u(a) u(b) + u(b) u(a) = 0$, or $u(a) u(b)  = - u(b) u(a)$.   Using this, 
if $w_1 = \frac{1}{2} (u(a) + i u(b))$ then a simple computation shows that 
$w_1 w_1^* w_1 = w_1$, so that $w_1$ is a partial isometry.    Let $w = z/2$.
 Clearly $\Vert w \Vert \leq 1$, but now we see that $$1 = \Vert w_1 \Vert = \Vert u(a)^2 w \Vert \leq \Vert w \Vert.$$   So $\Vert w \Vert = 1$ and $\Vert z \Vert = \Vert x \Vert = 2$.   This proves the last assertion of the theorem, 
since $\Vert {\rm Re}(x) \Vert = \Vert {\rm Im}(x) \Vert \leq 1$ by the maximality property of $\theta$,
but they clearly cannot be strict contractions since $\Vert x \Vert = 2$.   So henceforth we may assume that
$\theta = 0$ and $z = x$.

   Let $w_2 =  \frac{1}{2} (a_\perp + i \, b_\perp)$, 
so that $w = w_1 + w_2$, and $w_1 w_2 = w_2 w_1 = w_1 w_2^* = w_2^* w_1 = 0$.  
Note that $u(w) = w_1 + u(w_2) \neq 0$, and  $u(e^{-i \theta} w) = e^{-i \theta}u(w)$,
so $\Vert x \Vert = 2$.  Also 
$w w^* = w_1 w_1^* + w_2 w_2^*$.   Suppose $\psi$ is a state with $\psi(w_2 w_2^*) = 1$.
Then  since $\Vert w w^* \Vert \leq 1$ we must have $\psi(w_1 w_1^*) = 0$,
which forces $\psi(u(a)^2) =  \psi(u(b)^2) = 0$. 
 Thus $\psi \notin \{ u(a^2) \} ' = \{ a^2 \}'$ (see \cite[Lemma 3.3 (i)]{ER4}), so that $\psi(a_\perp^2) \neq 1$ since 
$a^2 = u(a^2) + a_\perp^2$.  
On the other hand, 
$$1 = \psi(\frac{a_\perp^2}{4} + \frac{b_\perp^2}{4} + i(\frac{b_\perp a_\perp}{4} - \frac{a_\perp b_\perp}{4})).$$
We deduce  the contradiction that
$\psi(a_\perp^2) = 1$. 
This contradiction shows that $1 - w_2 w_2^*$ is strictly positive so that 
$u(w_2 w_2^*) = 0$.   Hence $u(w w^*) = u(w_1 w_1^*) = w_1 w_1^*$, and so 
$u(w) = \lim_n \, (w w^*)^n \, w = w_1 w_1^* w = w_1$.    

Finally, suppose that $x \in A \cap  e^\perp B e^\perp$ has norm $1$ (so that $x$ may be taken to 
be our previous $w$).
Then $u(x)^2 =  w_1^2 = 0$.    That $x = u(x) + y$, where  $u(x)$ is orthogonal to $y$ and $y^*$, 
follows because $w = u(w) + w_2$ and $u(w) w_2 = u(w) w_2^* =w_2 u(w) = w_2^* u(w) = 0$, the latter because $u(w)$ is a
linear combination of the selfadjoint $u(a), u(b)$, which are each orthogonal to $a_\perp$ and $b_\perp$.
  \end{proof}

\begin{corollary} \label{m}  If the conditions of the previous lemmas hold, 
and also  $A \cap  e^\perp B e^\perp = (0)$,  then $P(A)$ is a subalgebra of $A$.
\end{corollary} 
\begin{proof}  For $x, y  \in P(A)$ we have $e x y e = e P(x y) e$ by Proposition \ref{ES}.
So  $xy - P(xy) \in e^\perp B^{**} e^\perp \cap A = (0)$.   Thus $P(A)$ is closed under products.  \end{proof}

\begin{corollary} \label{m4}  If the conditions of the previous lemmas hold, 
and also $B$ is commutative, then $A \cap  e^\perp B e^\perp = (0)$, and $P(A)$ is a subalgebra of $A$.
\end{corollary} 
\begin{proof}  By the proof of Lemma \ref{m3}, if $x \in A \cap  e^\perp B e^\perp$, and 
$e^{i \theta} x = a+ib$ with $a = a^*, b = b^*$ and $\Vert a \Vert = 1$, we obtained $u(a) u(b)  = - u(b) u(a) = 0$, and $u(a) = u(a) u(b)^2 = 0$.
This is impossible since  $\Vert a \Vert = 1$, so $A \cap  e^\perp B e^\perp = (0)$.  Then 
apply Corollary \ref{m}.  \end{proof} 

As in Section 4, to show $P(A)$ is a subalgebra of $A$, we may replace $A$ by $D$, the closed algebra generated
by $P(A)$.   After this is done, in the previous lemmas  $A \cap  e^\perp B e^\perp$ becomes
$(I-P)(D)$.  

\begin{theorem} \label{matre}
Let $P$ be a unital completely contractive projection on $A$ such that  $I-P$ is contractive on {\rm Re}($A$).  Suppose that  $A$ is a subalgebra of $M_N$ for some $N \in \Ndb$
and let $D$ be the closed algebra generated by $P(A)$.  Then every element of $(I-P)(D)$ is nilpotent. Furthermore, if $D$ is semisimple then the range of $P$ is a subalgebra of $A$.  
\end{theorem}

\begin{proof}  Note that $(I - P) (D)$ is an ideal of $D$ by Corollary \ref{newaei}. Suppose $x \in (I - P)(D)$ has norm 
$1$.   Suppose that $x$ is not nilpotent.  Set $y_1 = x$.
By Lemma \ref{m3}, $x = u(x) +x_{'}$ where $u(x)^{2}=0$ and $x_{'} \perp u(x)$.  Furthermore $x^{2} = (x_{'})^{2}$ lies in $D$ and $x^{2} \perp u(x)$.    Similarly, since $x^2 \neq 0$ we set $y_2 = x^{2}/\Vert x^2 \Vert$
then $y_2  = u(y_2)+(y_2)_{'}$ where $u(y_2)$ and $(y_2)_{'}$ are perpendicular to each other,
 and  $u(y_2)$  is perpendicular   to $u(x)$ (this
is because $u(y_2)$ is a limit of products beginning and ending with $x^2$, and e.g.\ 
$x^2 u(x) = (x_{'})^2 u(x) = 0$. Continuing in this way we obtain an infinite sequence of 
norm $1$ elements $y_k$ such that that $u(y_n) \perp u(y_k)$ for $k \leq n$.   It is well known that $u(y) \neq 0$ if $\Vert y \Vert = 1$.  This contradicts finite 
dimensionality.   So $x$ is nilpotent.  Since $(I-P)(D)$ is an ideal of $D$ consisting of
nilpotents, it follows that it lies in the Jacobson radical of $D$.   Thus if $D$ is semisimple then 
$(I - P) (D) = (0)$, so that $P(A) = D$ as before.
\end{proof}

For the following we no longer assume $A$ is finite dimensional but retain the other assumptions of the above Theorem.

\begin{lemma}   \label{slow}  Let $P$ be a unital completely contractive projection on $A$ such that  $I-P$ is contractive on {\rm Re}($A$), and let $D$ be the closed algebra generated by $P(A)$. If
 $x \in (I - P)(D)$ and $||x|| = 1$, then $||x^{2^{n}} || \le  \frac{2}{2^{2^{n}}}$.  Also, $x$ is quasi-regular (that is, quasi-invertible).
\end{lemma}

\begin{proof}  Let $x = a+ib$ as in previous Lemmas.  By Lemmas  \ref{m1} and \ref{m2}, $||a|| = ||b|| = 1/2$. Hence $||{\rm Re}(x^{2})|| = ||a^{2} - b^{2}|| \le 1/4$. It follows again from Lemma 5.1 that $||x^{2}|| \le 1/2$.  
The first result now follows by considering normalizations of further powers of 2, and using mathematical induction. It is easily seen that 
$$\sum_{k=1}^{\infty} || x^{k}|| \le 1 + \sum_{m=1}^{\infty} 2^{m} \, ||x^{2^{m}}|| \le 1 + \sum_{m=1}^{\infty} \, 2^{m} \, \frac{2}{2^{2^{m}}} < \infty.$$ It follows that $\sum_{k=1}^{\infty} x^{k}$ converges, so that
$1 - x$ is invertible.  \end{proof}

 {\bf Remark.}  It is still open whether the ideal $(I-P)(D)$ above consists entirely of
 quasi-regular elements.  If this is the case, then the above Theorem \ref{matre} holds for arbitrary unital operator algebras.    Note too that the assertion about  quasi-regulars in Lemma   \ref{slow} does follow from 
 Lemma \ref{m1}.  That result shows that the ideal $(I -P) (D)$ in $D$ has no nonzero real positive elements 
(for in the language of that result, if $a_{-} = 0$ then $a_{+} = b_+ = b_{-} = 0$). 
The ideas in the proof of Corollary 6.9 of \cite{BRII} then also show that if $x \in
{\rm Ball}((I-P)(D))$ then $x$ is quasi-regular.

\section{Jordan morphisms and Jordan subalgebras of operator algebras}

We recall that a {\em Jordan homomorphism} $T : A \to B$ is a linear map satisfying $T(ab+ba) = T(a) T(b) + T(b) T(a)$ for $a, b \in A$, or equivalently,
that $T(a^2) = T(a)^2$ for all $a \in A$ (the equivalence follows by applying $T$ to $(a+b)^2$).    By a {\em  Jordan operator algebra} we
shall simply mean a  norm-closed  {\em  Jordan subalgebra} $A$ of an operator algebra,  
namely a norm-closed   subspace closed under the `Jordan product' $\frac{1}{2}(ab+ba)$, or equivalently
 with $a^2 \in A$ for all $a \in A$ (that is, $A$ is
closed under squares).

It is natural to ask if the completely bicontractive algebra problem studied in Section 4
becomes simpler if the range of the projection $P : A \to A$ 
is also a Jordan subalgebra (that is $P(a)^2 \in P(A)$) for all $a \in A$.   We next dispense of this 
question: 

\begin{example}  \label{joup}   Let $y = E_{21} \oplus E_{12} \in M_4$, and let 
$$x = \left[ \begin{array}{ccccl} 0 & 0 & 0  & 1 \\
 0 & 0 & -1 & 0 \\  0 & 0 &  0 & 0 \\  0 & 0 &  0 & 0 
\end{array}
\right] . $$   Then $xy = -yx$, so that if $F$ is the span of $x$ and $y$ then $F$ is closed under squares.
However $F$ is not an algebra since $xy \notin F$. Let $V = \{ z \oplus z \in M_8 : z \in F \}$, and form the algebra
$A = A(V)$ described in Example \ref{Gco}.   By Corollary \ref{gex}  the canonical projection $P : A \to A$ which
replaces the $1$-$3$ entry of a matrix in $A(V)$ by $0$, 
 is  a (real completely positive)
 completely bicontractive and unital projection, but its range  is not a subalgebra.    However its range is a Jordan subalgebra; 
$P(A)$ is closed under squares since $z^2 = 0$ for $z \in F$.    Thus the  completely bicontractive algebra problem does not
become simpler if the range of the projection $P : A \to A$ 
is also a Jordan subalgebra.
\end{example}

The following variant of the `Banach--Stone theorem for $C^*$-algebras' will be evident to `JB-experts'.

\begin{lemma} \label{cic}  Let $A$ be a unital $C^*$-algebra,
and $T : A \to B(H)$ a unital complete isometry such that 
$T(A)$ is closed under taking squares (thus, $T(A)$ is a Jordan algebra).  Then $T(A)$ is a $C^*$-subalgebra
of $B(H)$, and $T$ is a $*$-homomorphism. \end{lemma}

\begin{proof}  Since such   $T$ is necessarily $*$-linear as we said in the introduction,  $T(A)$ is a JB*-algebra, hence a selfadjoint JB*-triple (see
e.g.\ \cite{Rod}).  By 
the theory of JB*-triples  $T$ is a Jordan homomorphism. (Two other proofs of this: look at the selfadjoint part and use 
the fact that isometries in that category are  Jordan homomorphisms \cite{IR};
or it can be deduced using the $C^*$-envelope as in the next proof).  In particular for each $x \in A_{\rm sa}$ we have
$T(x^2) = T(x)^2$, so by Choi's multiplicative domain result (see e.g.\ \cite[Proposition 1.3.11]{BLM})
we have $T(xy) = T(x) T(y)$ for all
$y \in A$.  So $T$ is a homomorphism and $T(A)$ is a  $C^*$-subalgebra.  \end{proof}

It is natural to ask if the analogous result is true for operator algebras.   That is,  if $B$ is a closed 
unital Jordan subalgebra of an 
operator  algebra  $A$, and if $B$ is unitally  and linearly completely 
isometric to another unital operator algebra, then is $B$ actually a subalgebra of $A$?
If the algebra is also commutative this is true and follows from the next result.

\begin{lemma} \label{cioa}  Let $A$ be a unital operator algebra,
and let $T : A \to B$ a unital complete isometry onto a unital Jordan operator algebra.  Then $T$ is a  Jordan homomorphism, and 
$T(a^n) = T(a)^n$ for every $n \in \Ndb$ and $a \in A$. \end{lemma}

\begin{proof}    Note that $T(a)^3$ is the Jordan product of $T(a)$ and $T(a)^2$, so 
$T(A)$ is closed under cubes.  Similarly it is closed under every power.    By the property of the $C^*$-envelope
mentioned in the introduction,
there exists a $*$-homomorphism $\pi : C^*(T(A)) \to C^*_e(A)$ with $\pi \circ T = I_A$.   
So $a^n = \pi(T(a)^n) = \pi(T(a^n))$.  Since $\pi_{\vert T(A)}$ is one-to-one, the results follow.
 \end{proof} 

{\bf Remark.}  In the proof of the last result one could have instead  used \cite[Corollary 2.8]{AS}.

\bigskip

We now answer the question above Lemma \ref{cioa}  in the negative: 

\begin{example}  Let $P :  A \to A$ be a completely contractive projection on an operator algebra  $A$ on $H$ whose kernel is an ideal $I$
(see e.g.\ Corollary \ref{newaei} or Lemma \ref{yoho}).   Then it is known that $B = A/I$ is an operator algebra
(see \cite[Proposition 2.3.4]{BLM}), and the induced map
$\tilde{P} : B \to  P(A)$ is a  completely isometric  isomorphism, and  $\tilde{P}$ will be unital
if $A$ and $P$ are unital.   If these hold, and in addition $P(A)$ is a Jordan subalgebra of $A$ which is not a subalgebra,
then  $T =  \tilde{P}: B \to A \subset B(H)$ is a unital complete isometry such that 
$T(B)$ is closed under taking squares (thus, $T(B)$ is a unital Jordan subalgebra), but  $T(B)$ is not a subalgebra, and $T$ is not
an algebra homomorphism.   In particular, we can take $A$ to be the algebra in Example 
\ref{joup}. \end{example}

We finish our paper with another `Banach--Stone type theorem for operator algebras': 

\begin{proposition} \label{bsoa}    Suppose that $T : A \to B$ is a  completely isometric surjection between approximately unital operator algebras.  Then $T$ is  real (completely)
positive if and only if $T$ is an algebra homomorphism.
 \end{proposition}  

  \begin{proof}   If $T$ is an algebra homomorphism then by Meyer's theorem  \cite[Theorem 2.1.13]{BLM}
$T$ extends to a unital completely isometric surjection between the unitizations, which then 
extends by Wittstock's extension theorem to a unital completely contractive, hence completely positive, map on a generated $C^*$-algebra.
So  $T$ is real completely positive.  

Conversely, suppose that  $T$ is real  positive.  We may assume that $A$ and $B$ are  unital, 
since $T^{**}$ is
a real positive   completely isometric surjection between unital 
operator algebras.   
By the Banach--Stone theorem for operator algebras \cite[Theorem 4.5.13]{BLM},
   there exists a completely isometric surjective homomorphism $\pi : A \to B$ and a unitary $u$
 with $u, u^{-1} \in B$ with $T = 
\pi(\cdot) u$.    The restriction of $T$ to $\Cdb 1$ is real positive, hence 
 positive (see \cite[Section 2]{BBS}).  Thus $u \geq 0$, and so $u = (u^2)^{\frac{1}{2}} = 1$.  Hence
$T$ is an algebra homomorphism.  
\end{proof}

{\bf Remark.}  One may also prove a limited version of this result for algebras with no kind of approximate identity by using the ideas in the proof of Proposition \ref{reau}.

\bigskip

There is also a  Jordan variant of the last result, a simple adaption of the main theorem in \cite{AS}.
Here we just state the unital case (see \cite{BNj} for more on this topic).

\begin{proposition} \label{bsoa2}    Suppose that $T : A \to B$ is an isometric surjection between unital Jordan operator algebras.  Then $T$ is real 
positive if and only if $T$ is a Jordan algebra homomorphism.
 \end{proposition}  

  \begin{proof}    If $T$ is a Jordan algebra homomorphism and $u = T(1_A)$ and $T(v) = 1_B$
then $2u = u 1_B + 1_B u  = 2T(1_A \cdot v) = 2T(v) = 2 \cdot 1_B$.   So  $u = 1_B$.  
However a unital contractive map is real positive \cite[Section 2]{BBS}.  

The converse follows by the same proof as for Proposition \ref{bsoa}, but using the form of 
the Banach--Stone theorem for operator algebras in \cite[Corollary 2.10]{AS}.  By that result 
   $T(1)$ is a unitary $u$
 with $u, u^{-1} \in B$.  Moreover $T(\cdot) u^{-1}$  is an isometric surjection onto $B$, 
so by the same result it it is a Jordan homomorphism $\pi$.   We have  $T = 
\pi(\cdot) u$, and we finish as before.
\end{proof}

{\em Acknowledgements.}  
Some of this material was presented at the
AMS Special Session on Operator Algebras and Their Applications: A Tribute to Richard V. Kadison, January 2015.
In the first authors article \cite{B2015} for the conference proceedings for this AMS Special Session
there is a very brief survey of some of the
material in the present paper.
Corollary \ref{san1} and part of Theorem  \ref{rcpua}  were stated in \cite{B2015} first, and later added to the
present paper.   We thank Joel Feinstein for a conversation during which we together arrived at an insight needed 
in these results.     We are also indebted to the referee for suggesting that we consider the Jordan algebra
version of the topics of this paper; this insight led to the work \cite{BNj} in progress.


\begin{thebibliography}{99} 
\bibitem{Ake2}   C. A. Akemann, {\em Left ideal structure of $C\sp*$-algebras,}  J.\ Funct.\ Anal.\
\textbf{6}  (1970), 305-317.


\bibitem{AS}   J. Arazy and B. Solel, {\em 
Isometries of nonselfadjoint operator algebras,} 
J. Funct. Anal. {\bf 90} (1990), no. 2, 284--305. 

\bibitem{Ar}   W. B. Arveson,
{\em Subalgebras of }$C^{*}-${\em algebras,} Acta Math. {\bf 123 }(1969), 141--224.

\bibitem{BaTim}  T. Barton and R. Timoney, {\em
Weak*-continuity of Jordan triple products and its
applications} Math. Scand. {\bf 59} (1986), 177--191.

 \bibitem{Bat}   M. Battaglia, {\em Order theoretic type decomposition
of JBW*-triples,}  Quart. J. Math. Oxford, {\bf 42} (1991), 129--147.



\bibitem{BBS}  C. A. Bearden, D. P. Blecher and S. Sharma, {\em On positivity and roots in operator algebras,}  Integral Equations Operator Theory {\bf  79} (2014), no. 4, 555--566.


\bibitem{Bare}  D. P. Blecher,  {\em Are operator algebras Banach algebras?}  In: Banach algebras and their applications, p.\ 53–58, Contemp. Math., 363, Amer. Math. Soc., Providence, RI, 2004.



\bibitem{Bnpi}  D.\  P.\ Blecher, {\em Noncommutative peak interpolation
revisited},  Bull.\ London Math.\ Soc. {\bf 45} (2013),
1100--1106.

\bibitem{B2015}  D. P. Blecher,  {\em Generalization of C*-algebra methods via real positivity for operator  and Banach algebras,}  Preprint 2015, to appear.


 
\bibitem{BHN}  D. P. Blecher, D. M. Hay, and
M. Neal, {\em Hereditary subalgebras of operator algebras,} J.\
Operator Theory {\bf 59} (2008), 333-357.

\bibitem{BL}  D. P. Blecher
and   L. E. Labuschagne, {\em Logmodularity and isometries of operator algebras,} Trans. Amer.
Math. Soc. {\bf 355} (2002), 1621--1646.




\bibitem{BLM}  D. P. Blecher
and C.  Le Merdy, {\em Operator algebras and their modules---an
operator space approach,} Oxford Univ.\  Press, Oxford (2004).


\bibitem{BM}   D. P. Blecher and B. Magajna, {\em Duality and operator algebras II:
Operator algebras as Banach algebras,} Journal of Functional Analysis
{\bf 226} (2005), 485--493.

\bibitem{BNI} D. P. Blecher
and M. Neal, {\em Open projections in operator algebras I: Comparison Theory},
Studia Math. {\bf 208} (2012), 117--150.

\bibitem{BNII} D. P. Blecher
and M. Neal, {\em Open projections in operator algebras II: compact projections,} Studia Math.\ {\bf 209} (2012), 203--224.

\bibitem{BNj} D. P. Blecher
and M. Neal, {\em Contractive projections on Jordan operator algebras,}  in preparation.
 
 \bibitem{BOZ}   D. P. Blecher and N. Ozawa, {\em Real 
positivity and approximate identities
in Banach algebras}, Pacific J. Math.  {\bf 277} (2015), 1--59.

\bibitem{BRI}  D. P. Blecher and C. J. Read, {\em  Operator algebras with contractive approximate identities,}
J. Functional Analysis {\bf 261} (2011), 188--217.

\bibitem{BRII}  D. P. Blecher and C. J. Read, {\em  Operator algebras with contractive approximate identities II,} J. Functional Analysis {\bf 
264} (2013), 1049--1067.

\bibitem{BRord}  D. P. Blecher and C. J. Read, {\em  Order theory 
and interpolation in operator algebras}, Studia Math. {\bf  225} (2014), 61--95.

\bibitem{Rod}  M. Cabrera Garcia and A.  Rodriguez Palacios, {\em 
	Non-associative normed algebras. Vol. 1,}
The Vidav-Palmer and Gelfand-Naimark theorems. Encyclopedia of Mathematics and its Applications, 154. Cambridge University Press, Cambridge, 2014.

\bibitem{CECPL}  M.-D. Choi and E. G. Effros, {\em
The completely positive lifting problem for C∗-algebras,}  
Ann. of Math. {\bf 104} (1976), 585--609. 


\bibitem{CE}  M.-D. Choi and E. G. Effros, {\em Injectivity and
operator spaces,} J.\ Funct.\ Anal.\  {\bf 24}
(1977), 156--209.

\bibitem{CNR}  C.H. Chu, M. Neal and B. Russo, {\em  Normal contractive projections preserve type}, J. Operator Theory {\bf  51}
(2004), 281--301.



\bibitem{ER4}  C. M. Edwards and G. T. R\"uttimann, {\em Compact tripotents in
bi-dual JB*-triples,} Math. Proc. Camb. Philos. Soc. {\bf 120}  (1996), 155--173.

\bibitem{Ef}
E. G. Effros, {\em  Order ideals in a $C\sp{*} $-algebra and its dual,}  Duke Math. J.
{\bf 30}  (1963), 391--411.



\bibitem{ES}
E. G. Effros and E. St{\o}rmer, \emph{Positive projections and Jordan structure in operator algebras,} Math. Scand. 45 (1979),  127--138. 



\bibitem{CPOC}  Y. Friedman and B.  Russo, {\em Contractive projections
on $C_0(K)$}, Trans. AMS {\bf 273} (1982), 57--73.

  
\bibitem{FRAD}  Y. Friedman and B.  Russo, {\em
Conditional expectation without order,}  Pacific J. Math. 
{\bf  115} (1984), 351--360.

\bibitem{CPP}  Y. Friedman and B.  Russo, Solution of
the contractive projection problem, {\em  J.\ Funct.\ Anal.\ } {\bf 60}
(1985), 56--79.



\bibitem{GKS}  G. Godefroy,  N. J. Kalton and P. D. Saphar, {\em  Unconditional ideals in Banach spaces,}
Studia Math. {\bf 104} (1993), 13--59. 

\bibitem{Ham} M. Hamana, {\em Triple  envelopes and Silov boundaries of operator
spaces,}  Math. J. Toyama University {\bf 22} (1999), 77--93.

\bibitem{Har}  L. A.  Harris, {\em A generalization of C∗-algebras,}
Proc. London Math. Soc. (3) {\bf 42} (1981), no. 2, 331--361.


\bibitem{IR} J. M.  Isidro and A.  Rodriguez-Palacios, {\em Isometries of JB-algebras,} Manuscripta Math. {\bf 86} (1995), no. 3, 337--348. 

 \bibitem{Kad1}   R. V. Kadison, {\em  Isometries of
operator algebras,} Ann.\ of Math.\ {\bf 54} (1951), 325--338.

\bibitem{Kad2}   R. V. Kadison, {\em A generalized Schwarz inequality,} 
Ann.\ of Math.\  {\bf 56} (1952), 494--503.

\bibitem{LL} A. T. Lau and  R. J. Loy, {\em Contractive projections on Banach algebras,} J. Funct. Anal. {\bf 254} (2008), 2513--2533. 


\bibitem{N} M.  Neal, {\em Inner ideals and facial structure of the quasi-state space of a JB-algebra,}
 J. Funct.
Anal.\ {\bf  173} (2000), 284--307.

\bibitem{N2} M.  Neal, E. Ricard and B. Russo, {\em Classification of contractively complemented Hilbertian operator spaces,} J.\ Funct.\ Anal.\  237 (2006), 589--616.

\bibitem{NR00}  M. Neal and B. Russo, {\em Projecteurs contractifs et espaces operateurs,} C. R. Acad. Sci. Paris, t. 331, Serie I (2000),
873--878.

\bibitem{NR03}  M. Neal and B. Russo, {\em Contractive projections and operator spaces}, Trans. Amer. Math. Soc. {\bf 355} (2003), 2223--2262. 


\bibitem{NeRu} M. Neal and B. Russo, {\em Operator space
characterizations of
C*-algebras and ternary rings}, Pacific J. Math. {\bf 209} (2003), 339--364.


\bibitem{NR05}  M. Neal and B. Russo, {\em Representation of contractively complemented Hilbertian operator spaces on the Fock space,} Proc.\  Amer. Math. Soc. {\bf  134}, No. 2, (2005), 475-485


\bibitem{NR}   M. Neal and B. Russo, {\em  Existence of contractive projections on preduals of JBW∗-triples,}
 Israel J. Math. {\bf 182} (2011), 293--331. 

\bibitem{Paul}  V. I. Paulsen,   {\em Completely bounded maps and operator
algebras,} Cambridge Studies in Advanced Math., 78, Cambridge
University Press, Cambridge, 2002.
 

\bibitem{P} G. K. Pedersen, {\em C*-algebras and their automorphism
groups,} Academic Press, London (1979).



\bibitem{RY} A. G. Robertson and M.A. Youngson, {\em Positive projections with contractive complements on Jordan algebras,} J. London Math. Soc. (2) 25 (1982),  365--374. 

\bibitem{SOJ}  B. Russo, {\em Structure of JB*-triples,} In: Jordan
Algebras,
 Proc. of Oberwolfach Conference 1992, Ed. W. Kaup, et al., de Gruyter
(1994), 209--280.



\bibitem{St}  E. St{\o}rmer, {\em Positive projections with contractive complements on $C^*$-algebras,} J. London Math. Soc. (2) 26 (1982),  132--142. 



\bibitem{Y}  M. A. Youngson, {\em Completely contractive projections on C∗-algebras}, Quart. J. Math. Oxford Ser. (2) 34 (1983), 507--511.

\end{thebibliography}
\end{document}